\documentclass[a4paper,10pt]{amsart}
\usepackage[matrix,arrow,tips,curve]{xy}
\usepackage[english]{babel}
\usepackage{amsmath}
\usepackage{amssymb}
\usepackage{tikz}
\usepackage{mathrsfs}
\usepackage{enumerate}
\usepackage{graphicx}
\usepackage{hyperref}
\usetikzlibrary{trees}
\usetikzlibrary{arrows}

\input{xy}
\xyoption{all}
\newtheorem{thm}{Theorem}[section]
\newtheorem{Lemma}[thm]{Lemma}
\newtheorem{Proposition}[thm]{Proposition}
\newtheorem{Corollary}[thm]{Corollary}

\newtheorem*{thm*}{Theorem}

\newtheorem{theoremn}{Theorem}

\theoremstyle{definition}

\newtheorem{Construction}[thm]{Construction}
\newtheorem{Definition}[thm]{Definition}
\newtheorem{Remark}[thm]{Remark}

\newtheorem{Example}[thm]{Example}

\newtheorem{say}[thm]{}

\newcommand{\Aut}{\operatorname{Aut}}
 
\renewcommand{\P}{\mathbb{P}}
\newcommand{\cM}{\overline{\mathcal{M}}}

\DeclareMathOperator{\Cl}{Cl}
\DeclareMathOperator{\NE}{NE}
\DeclareMathOperator{\Pic}{Pic}

\DeclareMathOperator{\Eff}{Eff}

\DeclareMathOperator{\mult}{mult}

\DeclareMathOperator{\Exc}{Exc}

\DeclareMathOperator{\Sing}{Sing}

\DeclareMathOperator{\ldim}{ldim}
\DeclareMathOperator{\Vol}{Vol}

\DeclareMathOperator{\Cone}{Cone}
\DeclareMathOperator{\Nef}{Nef}
\DeclareMathOperator{\Mov}{Mov}

\newcommand\Span[1]{\langle{#1}\rangle}
\newcommand{\quot}{/\hspace{-1.2mm}/}

\hypersetup{pdfpagemode=UseNone}
\hypersetup{pdfstartview=FitH}
		
\begin{document}
\title[\resizebox{5.5in}{!}{Birational geometry of moduli spaces of configurations of points on the line}]{Birational geometry of moduli spaces of configurations of points on the line}

\author[Michele Bolognesi]{Michele Bolognesi}
\address{\sc Michele Bolognesi\\
IMAG - Universit\'e de Montpellier\\
Place Eug\`ene Bataillon\\
34095 Montpellier Cedex 5\\ France}
\email{michele.bolognesi@umontpellier.fr}

\author[Alex Massarenti]{Alex Massarenti}
\address{\sc Alex Massarenti\\ Dipartimento di Matematica e Informatica, Universit\`a di Ferrara, Via Machiavelli 30, 44121 Ferrara, Italy}
\email{alex.massarenti@unife.it}

\date{\today}

\begin{abstract}
In this paper we study the geometry of GIT configurations of $n$ ordered points on $\mathbb{P}^1$ both from the the birational and the biregular viewpoint. In particular, we prove that any extremal ray of the Mori cone of effective curves of the quotient $(\mathbb{P}^1)^n\quot PGL(2)$, taken with the symmetric polarization, is generated by a one dimensional boundary stratum of the moduli space. Furthermore, we develop some technical machinery that we use to compute the canonical divisor and the Hilbert polynomial of $(\mathbb{P}^1)^n\quot PGL(2)$ in its natural embedding, and its automorphism group.
\end{abstract}

\maketitle
\tableofcontents

\section*{Introduction}
The aim of this paper is to explicitly study, both from the birational and the projective viewpoint, the compactification of $\mathcal{M}_{0,n}$ given by the GIT quotient $\Sigma_m : = (\mathbb{P}^1)^{m+3}\quot PGL(2)$ of configurations of $n = m+3$ ordered points on $\P^1$ with respect to the symmetric polarization on $(\mathbb{P}^1)^n$. The aim of our notation is to stress the dimension of the space. 

For instance, $\Sigma_2$ is a del Pezzo surface of degree five, it is well-known that $\Sigma_2$ can be embedded in $\mathbb{P}^5$ as a smooth surface of degree five that is a blow-up of $\mathbb{P}^2$ in four general points, the Mori cone of $\Sigma_2$ is generated by the exceptional divisor  and the strict transforms of the lines through two of the blown-up points, and $\Aut(\Sigma_2)\cong S_5$. The next GIT quotient $\Sigma_3$ can be embeded in $\mathbb{P}^4$ as the Segre cubic \cite{Do15}, that is the unique, up to automorphisms of $\mathbb{P}^4$, cubic hypersurface in $\mathbb{P}^4$ with ten nodes, and $\Aut(\Sigma_3)\cong S_6$. Our aim is to generalize these classical results in arbitrary dimension. 

When $m$ is even the variety $\Sigma_m$ is smooth. Furthermore, it is a Mori dream space. Mori dream spaces, introduced in \cite{HK00}, are roughly speaking normal projective varieties behaving very well with respect to the minimal model program. Indeed the birational geometry of a Mori dream space can be encoded in some finite data, namely its cone of effective divisors together with a chamber decomposition on it called the Mori chamber decomposition of the effective cone. We refer to Section \ref{MDS} and the references therein for the rigorous definition and special properties of Mori dream spaces. By \cite[Section 2]{HK00} the Mori chamber decomposition of $\Sigma_m$ may be read off from the VGIT decomposition of the ample cone of $(\mathbb{P}^1)^{m+3}$. However, in this paper we prefer to adopt an explicit birational approach in the spirit of what has been done in \cite{AM16} for blow-ups of projective spaces, and in \cite{AC16} for the Fano variety of linear spaces contained in the complete intersection of two odd-dimensional quadrics.

In one of the most celebrated papers \cite{DM69} in the history of algebraic geometry P. Deligne and D. Mumford proved that there exists an irreducible scheme $\mathcal{M}_{g,n}$ coarsely representing the moduli functor of $n$-pointed genus $g$ smooth curves. Furthermore, they provided a compactification $\overline{\mathcal{M}}_{g,n}$ of $\mathcal{M}_{g,n}$ adding the so-called Deligne-Mumford stable curves as boundary points. Afterwards other compactifications of $\mathcal{M}_{g,n}$ have been introduced, see for instance \cite{Ha}. 

The moduli space $\cM_{0,n}$ is stratified by topological type as follows: a codimension $1$-stratum is an irreducible component of the locus of $\cM_{0,n}$ given by curves having at least one node. Thus the general element of a codimension $1$-stratum has two irreducible components. The resulting divisor is called a boundary divisor. We continue increasing the number of nodes of the general curve of the component until reaching the dimension $1$-strata, which are irreducible components of loci parametrizing curves with at least $n - 4$ nodes. These $1$-dimensional boundary strata of $\cM_{0,n}$ are called F-curves.

The F-conjecture asserts that any extremal ray of the Mori cone $\NE(\cM_{0,n})$ is generated by a one dimensional boundary stratum \cite[Conjecture 0.2]{GKM02}. So far the conjecture for $\cM_{0,n}$ is known up to $n=7$ \cite[Theorem 1.2 (3)]{KMc}, \cite{Lars}, see \cite{FGL16} for a more recent general account, and the very intricate combinatorics of the moduli space $\cM_{0,n}$ for higher $n$ seem an obstacle pretty difficult to avoid if one wants to try his chance.

In this paper we prove an analogous statement for $\Sigma_m$. We tackle this problem taking advantage of the fact that $\Sigma_m$ is a Mori Dream Space, while we know that $\cM_{0,n}$ is not a Mori Dream Space for $n\geq 10$ \cite[Corollary 1.4]{CT15}, \cite[Theorem 1.1]{GK16}, \cite[Addendum 1.4]{HKL16}. Clearly, this makes the Mori cone of $\Sigma_m$ much more manageable than the Mori cone of $\overline{\mathcal{M}}_{0,n}$. 

Recall that the quotient $\Sigma_m$ gives an alternate compactification of $\mathcal{M}_{0,n}$, which is a little coarser than $\cM_{0,n}$ on the boundary, and in fact it is the target of a birational morphism $\cM_{0,n} \to \Sigma_m$, that we will recall in the body of the paper, see also \cite{Bo11}. 

Nevertheless, also $\Sigma_m$ has a stratification of its boundary locus, similar to that of $\cM_{0,n}$, and one can ask exactly the same question of the F-Conjecture.

By Proposition \ref{picpari} if $n = m+3$ is even then $\Pic(\Sigma_m)\cong \mathbb{Z}$ and its cones of curves and divisors are $1$-dimensional. This result could be obtained from Kempf's descent lemma \cite[Theorem 2.3]{DN89} but again we prefer to give an explicit birational proof.
 
On the other hand, if $n = m+3$ is odd then $\Pic(\Sigma_m)\cong \mathbb{Z}^{m+3}$ and the birational geometry of $\Sigma_m$ gets more interesting. In this case we manage to describe its Mori chamber decomposition and consequently the cones of nef and effective divisors of $\Sigma_m$ and their dual cones of effective and moving curves. For instance, we prove the following:

\begin{theoremn}\label{inthm}
The Mori cone $\NE(\Sigma_m)$ of the GIT quotient $\Sigma_m : = (\mathbb{P}^1)^{m+3}\quot PGL(2)$ of configurations of $n = m+3$ ordered points on $\P^1$ with respect to the symmetric polarization on $(\mathbb{P}^1)^n$ is generated by classes of $1$-dimensional boundary strata for any $m\geq 1$.
\end{theoremn}

In Theorem \ref{ful} we will also describe precisely what are these $1$-dimensional strata. The main ingredients of the proof of Theorem \ref{inthm} is a construction of our GIT quotients as images $\Sigma_m\subset\mathbb{P}^N$ of rational maps induced by certain linear systems on the projective space $\mathbb{P}^m$, due to C. Kumar \cite{Ku00, Ku03}, a careful analysis of the Mori chamber decomposition of the movable cone of certain blow-ups of the projective space and some quite refined projective geometry of the GIT quotients. 

More precisely C. Kumar realized $\Sigma_m$ as the closure of the image of the rational map induced by the linear system $\mathcal{L}_{2g-1}$ of degree $g$ hyersurfaces of $\mathbb{P}^{2g-1}$ having multiplicity $g-1$ at $2g+1$ general points if $m = 2g-1$ is odd, and as the closure of the image of the rational map induced by the linear system $\mathcal{L}_{2g}$ of degree $2g+1$ hypersurfaces of $\mathbb{P}^{2g}$ having multiplicity $2g-1$ at $2g+2$ general points if $m = 2g$ is even. In particular, $N = h^{0}(\mathbb{P}^{2g-1},\mathcal{L}_{2g-1})$ if $m = 2g-1$ is odd, and $N = h^{0}(\mathbb{P}^{2g},\mathcal{L}_{2g})$ if $m = 2g$ is even.

Thanks to recent results on the dimension of linear systems on the projective space due to M. C. Brambilla, O. Dumitrescu and E. Postinghel \cite{BDP16}, we obtain some explicit formulas for the Hilbert polynomial of $\Sigma_m\subset\mathbb{P}^N$. These results are resumed in Corollaries \ref{hilbseg}, \ref{hilbodd}. We would like to mention that an inductive formula for the degree of these GIT quotients had already been given in \cite{HMSV09}, while a closed formula for the Hilbert function of GIT quotients of evenly weighted points on the line had been given in \cite{HH14}. The arguments in \cite{HH14} are based on a representation theoretical study of the coordinate ring of $\Sigma_m$ while our approach, which works in the oddly weighted case as well, privileges the geometrical viewpoint. Finally, thanks to a detailed description of the projective geometry of $\Sigma_m$ we manage to compute its automorphism groups. 

The main results on the geometry of $\Sigma_m\subset\mathbb{P}^N$ in Sections \ref{oddismi}, \ref{evenio}, \ref{secaut}, and Corollaries \ref{hilbseg}, \ref{hilbodd} can be summarized as follows.

\begin{theoremn}
Let us consider the GIT quotient $\Sigma_m\subset\mathbb{P}^N$. If $m = 2g-1$ is odd we have 

$$\Pic(\Sigma_{2g-1})\cong\mathbb{Z},\ K_{\Sigma_{2g-1}}\cong\mathcal{O}_{\Sigma_{2g-1}}(-2)$$

 and the  
the Hilbert polynomial of $\Sigma_{2g-1}\subseteq \mathbb{P}^N$ is given by 

$$h_{\Sigma_{2g-1}}(t) = \binom{gt+2g-1}{2g-1}+\sum_{r=0}^{g-2}(-1)^{r+1}\binom{2g+1}{r+1}\binom{t(g-r-1)+2g-1-r-1}{2g-1}$$

In particular 
$$\deg(\Sigma_{2g-1})= g^{2g-1}+\sum_{r=0}^{g-2}(-1)^{r+1}\binom{2g+1}{r+1}(g-r-1)^{2g-1}$$

If $m = 2g$ is even we have 

$$\Pic(\Sigma_{2g})\cong\mathbb{Z}^{2g+3},\ K_{\Sigma_{2g}}\cong\mathcal{O}_{\Sigma_{2g}}(-1)$$

and the  
the Hilbert polynomial of $\Sigma_{2g}\subseteq \mathbb{P}^N$ is given by 

$$h_{\Sigma_{2g}}(t) = \binom{(2g+1)t+2g}{2g}+\sum_{r=0}^{g-1}(-1)^{r+1}\binom{2g+2}{r+1}\binom{t(2g-2r-1)+2g-r-1}{2g}$$

In particular 
$$\deg(\Sigma_{2g})= (2g+1)^{2g}+\sum_{r=0}^{g-1}(-1)^{r+1}\binom{2g+2}{r+1}(2g-2r-1)^{2g}$$

Furthermore, the automorphism group of $\Sigma_m$ is isomorphic to the symmetric group on $n = m+3$ elements $S_n$ for any $m\geq 2$. 
\end{theoremn}

Though a VGIT approach might be adopted in dealing with this kind of problems throughout the paper we will privilege explicit arguments of birational and projective nature.

\subsection*{Plan of the paper}
The paper is organized as follows. In Section \ref{pre} we recall some well-known facts and prove some preliminary results on GIT quotient, moduli spaces of weighted pointed rational curves and we clarify the relations between them. In Section \ref{secbir} we describe the Mori cone of $\Sigma_m$ with $m = 2g$ even. In Section \ref{degrees} we work out explicit formulas for the Hilbert polynomial and the degree of $\Sigma_m\subset\mathbb{P}^N$. Finally, in Section \ref{secaut} we compute the automorphism groups of $\Sigma_m$.

\subsection*{Acknowledgments}
The authors are members of the Gruppo Nazionale per le Strutture Algebriche, Geometriche e le loro Applicazioni of the Istituto Nazionale di Alta Matematica "F. Severi" (GNSAGA-INDAM). The first named author is member of the GDR GAGC of the CNRS. We thank the referees for the helpful comments that helped us to improve the paper, and Samuel Payne and Rohini Ramadas for pointing out a mistake in Theorem \ref{moving}.

\section{GIT quotients of $(\mathbb{P}^1)^n$}\label{pre}
The main characters of the paper are the GIT quotients $(\P^1)^n\quot  PGL(2)$, that we review now very quickly. For details there is plenty of very good references on this subject \cite{Do03,Do12,HMSV09,Bo11,DO88}. The application of GIT quotients to extract information about the F-conjecture or close analogs is a philosophy that dates back to the thesis of M. Simpson \cite{Si08} and papers \cite{GS11,BG15}, in the same spirit.

Let us consider the diagonal $PGL(2)$-action on $(\mathbb{P}^1)^n$. An ample line bundle $L$ endowed
with a linearization for the $PGL(2)$-action is called a polarization. Such a polarization on $(\mathbb{P}^1)^n$
is completely determined by an $n$-tuple $b = (b_1,\dots,b_n)$ of positive integers: 
$$L = \boxtimes_{i=1}^n\mathcal{O}_{\mathbb{P}^1}(b_i)$$
Now let us set $|b| = b_1+\dots+b_n$. A point $x\in (\mathbb{P}^1)^n$ is said to be $b$-semistable if for some $k > 0$, there exists a $PGL(2)$-invariant section $s \in H^0((\mathbb{P}^1)^n,L^{\otimes k})^{PGL(2)}$ such that $X_s:=\{ y\in (\mathbb{P}^1)^n: s(y)\neq 0\}$ is affine and contains $x$. A semistable point $x \in (\mathbb{P}^1)^n$ is stable if its stabilizer under the $PGL(2)$ action is finite and all the orbits of $PGL(2)$ in $X_s$ are closed. A categorical quotient of the open set $((\mathbb{P}^1)^n)^{ss}(b)$ of semistable points exists, and this is what we normally denote by $X(b)\quot PGL(2)$. We will omit to specify the polarization when it is $(1,\dots,1)$. Recall that, if $|b|$ is odd, then $H^0((\mathbb{P}^1)^n,L^{\otimes k})^{PGL(2)}=0$ for odd $k$, and 
\stepcounter{thm}
\begin{equation}\label{perdue}
(\mathbb{P}^1)^n(b)\quot PGL(2) = (\mathbb{P}^1)^n(2b)\quot PGL(2)
\end{equation}

\stepcounter{thm}
\subsection{Linear systems on $\mathbb{P}^n$}
Throughout the paper we will denote by $\mathcal{L}_{n,d}(m_1,\dots,m_s)$  the linear system of hypersurfaces of degree $d$ in $\mathbb{P}^n$ passing through $s$ general points $p_1,\dots,p_s\in\mathbb{P}^n$ with multiplicities respectively $m_1,\dots,m_s$. It was pointed out by Kumar \cite[Section 3.3]{Ku03} that the GIT quotients $(\mathbb{P}^1)^n(b)\quot PGL(2)$ can be obtained as the images of certain rational polynomial maps defined on $\P^{n-3}$.

\begin{thm}\cite[Theorem 3.4]{Ku03}\label{Kum}
Let us assume that $b_i < \sum_{j\neq i}b_j$ for any $i=1,\dots,n$, and let $p_1,\dots,p_{n-1}\in\mathbb{P}^{n-3}$ be general points. 
\begin{itemize}
\item[-] If $|b| = 2\overline{b}$ is even let 
$$\mathcal{L}=\mathcal{L}_{n-3,\overline{b}-b_n}(\overline{b}-b_n-b_1,\dots,\overline{b}-b_n-b_i,\dots,\overline{b}-b_n-b_{n-1})$$ 
be the linear system of degree $\overline{b}-b_n$ hypersurfaces in $\mathbb{P}^{n-3}$ with multiplicity $\overline{b}-b_i-b_n$ at $p_i$. 
\item[-] If $|b|$ is odd let 
$$\mathcal{L}=\mathcal{L}_{n-3,b-2b_n}(b-2b_n-2b_1,\dots,b-2b_i-2b_n,\dots, b-2b_n-2b_n)$$ 
be the linear system of degree $b-2b_n$ hypersurfaces in $\mathbb{P}^{n-3}$ with multiplicity $b-2b_i-2b_n$ at $p_i$. 
\end{itemize}
Finally, let $\phi_{\mathcal{L}}:\mathbb{P}^{n-3}\dasharrow \mathbb{P}(H^0(\mathbb{P}^{n-3},\mathcal{L})^{*})$ be the rational map induced by $\mathcal{L}$. Then $\phi_{\mathcal{L}}$ maps birationally $\mathbb{P}^{n-3}$ onto $(\mathbb{P}^1)^n(b)\quot PGL(2)$. In other words $(\mathbb{P}^1)^n(b)\quot PGL(2)$ may be realized as the closure of the image $\phi_{\mathcal{L}}$ in $\mathbb{P}(H^0(\mathbb{P}^{n-3},\mathcal{L})^{*})$.
\end{thm}

\begin{Example}
For instance, if $n=6$ and $b_1=\dots=b_6=1$ then $\overline{b} = 3$, and $\mathcal{L}=\mathcal{L}_{3,2}(1,\dots,1)$. In this case the rational map $\phi_{\mathcal{L}}$ is given by the quadrics in $\mathbb{P}^3$ passing through five general points, and $(\mathbb{P}^1)^n(b)\quot PGL(2)\subset\mathbb{P}^4$ is the Segre cubic $3$-fold. This is a very well-known classical object, see \cite{Do15,AB15,BB12} for some historic perspective and applications.

If $n=5$ and $b_1=\dots=b_5=1$ then $\mathcal{L}=\mathcal{L}_{2,3}(1,\dots,1)$ that is the linear system of plane cubics through four general points. In this case the quotient is a del Pezzo surface of degree five.
\end{Example}

\stepcounter{thm}
\subsection{Moduli of weighted pointed curves}
In \cite{Ha}, B. Hassett introduced moduli spaces of weighted pointed curves. 
Given $g\geq 0$ and rational weight data  $A[n] = (a_{1},\dots,a_{n})$, $0< a_{i}\leq 1$, satisfying $2g-2 + \sum_{i = 1}^{n}a_{i} > 0$,
the moduli space $\cM_{g,A[n]}$ parametrizes  genus $g$ nodal $n$-pointed curves $\{C,(x_1,\dots,x_n)\}$ 
subject to the following stability conditions:
\begin{itemize}
\item[-] each $x_i$ is a smooth point of $C$, and the points $x_{i_1}, \dots,  x_{i_k}$ are allowed to coincide only if $\sum_{j= 1}^{k}a_{i_j}\leq 1$,
\item[-] the twisted dualizing sheaf $\omega_C(a_{1}x_1 +\dots + a_{n}x_n)$ is ample. 
\end{itemize}
In particular,  $\cM_{g,A[n]}$ is a compactifications of the moduli space $\mathcal{M}_{g,n}$ of genus $g$ smooth $n$-pointed curves. 

\begin{say}\label{reduction}
For fixed $g,n$, consider two collections of weight data $A[n],B[n]$ such that $a_i\geq b_i$ for any $i = 1,\dots,n$. Then there exists a birational \textit{reduction morphism}
$$\rho_{B[n],A[n]}:\cM_{g,A[n]}\rightarrow\cM_{g,B[n]}$$
associating to a curve $[C,s_1,\dots,s_n]\in\cM_{g,A[n]}$ the curve $\rho_{B[n],A[n]}([C,s_1,\dots,s_n])$ obtained by collapsing components of $C$ along which $\omega_C(b_1s_1+\dots+b_ns_n)$ fails to be ample, where $\omega_C$ denotes the dualizing sheaf of $C$.
\end{say}

\begin{say}\label{forgetful}
Furthermore, for any $g$, consider a collection of weight data $A[n]=(a_1,\dots,a_n)$ and a subset $A[r]:=(a_{i_{1}},\dots,a_{i_{r}})\subset A[n]$ such that $2g-2+a_{i_{1}}+\dots+a_{i_{r}}>0$. Then there exists a \textit{forgetful morphism} 
$$\pi_{A[n],A[r]}:\cM_{g,A[n]}\rightarrow\cM_{g,A[r]}$$
associating to a curve $[C,s_1,\dots,s_n]\in\cM_{g,A[n]}$ the curve $\pi_{A[n],A[r]}([C,s_1,\dots,s_n])$ obtained by collapsing components of $C$ along which $\omega_C(a_{i_{1}}s_{i_{1}}+\dots+a_{i_{r}}s_{i_{r}})$ fails to be ample.
\end{say}
One of the most elegant aspects of the theory of rational pointed curves is the relation with rational normal curves and their projective geometry. This has been outlined by Kapranov in \cite{Ka}. Here below we briefly recall this, and his construction of $\cM_{0,n}$ as an iterated blow-up of $\P^{n-3}$.

\stepcounter{thm}
\subsubsection*{Kapranov's blow-up construction}
We follow \cite{Ka}. Let $(C,x_{1},\dots,x_{n})$ be a genus zero $n$-pointed stable curve. The dualizing sheaf $\omega_{C}$ of $C$ is invertible \cite{Kn}. By \cite[Corollaries 1.10 and 1.11]{Kn} the sheaf $\omega_{C}(x_{1}+\dots+x_{n})$ is very ample and has $n-1$ independent sections. Then it defines an embedding $\phi:C\rightarrow\mathbb{P}^{n-2}$. In particular, if $C\cong\mathbb{P}^{1}$ then $\deg(\omega_{C}(x_{1}+\dots+x_{n})) = n-2$, $\omega_{C}(x_{1}+\dots+x_{n})\cong\phi^{*}\mathcal{O}_{\mathbb{P}^{n-2}}(1)\cong\mathcal{O}_{\mathbb{P}^{1}}(n-2)$, and $\phi(C)$ is a degree $n-2$ rational normal curve in $\mathbb{P}^{n-2}$. By \cite[Lemma 1.4]{Ka} if $(C,x_{1},\dots,x_{n})$ is stable the points $p_{i} = \phi(x_{i})$ are in linear general position in $\mathbb{P}^{n-2}$.

This fact combined with a careful analysis of limits in $\cM_{0,n}$ of $1$-parameter families contained in $\mathcal{M}_{0,n}$ are the key for the proof of the following theorem \cite[Theorem 0.1]{Ka}.

\begin{thm}\label{kaphilb}
Let $p_{1},\dots,p_{n}\in\mathbb{P}^{n-2}$ be points in linear general position, and let $V_{0}(p_{1},\dots,p_{n})$ be the scheme parametrizing rational normal curves through $p_{1},\dots,p_{n}$. Consider $V_{0}(p_{1},\dots,p_{n})$ as a subscheme of the Hilbert scheme $\mathcal{H}$ parametrizing subschemes of $\mathbb{P}^{n-2}$. Then
\begin{itemize}
\item[-] $V_{0}(p_{1},\dots,p_{n})\cong \mathcal{M}_{0,n}$.
\item[-] Let $V(p_{1},\dots,p_{n})$ be the closure of $V_{0}(p_{1},\dots,p_{n})$ in $\mathcal{H}$. Then $V(p_{1},\dots,p_{n})\cong\cM_{0,n}$. 
\end{itemize}
\end{thm}

Kapranov's construction allows to translate many questions about $\cM_{0,n}$ into statements on linear systems on $\P^{n-3}$. Consider a general line $L_{i}\subset\mathbb{P}^{n-2}$ through $p_{i}$. There exists a unique rational normal curve $C_{L_{i}}$ through $p_{1},\dots,p_{n}$, and with tangent direction $L_{i}$ in $p_{i}$. Let $[C,x_{1},\dots,x_{n}]\in\cM_{0,n}$ be a stable curve, and let $\Gamma\in V_0(p_{1},\dots,p_{n})$ be the corresponding rational normal curve. Since $p_{i}\in\Gamma$ is a smooth point, by considering the tangent line $T_{p_{i}}\Gamma$ we get a morphism 
\stepcounter{thm}
\begin{equation}\label{KapMor}
\begin{array}{cccc}
f_i: & \cM_{0,n} & \longrightarrow & \mathbb{P}^{n-3}\\
 & [C,x_{1},\dots,x_{n}] & \longmapsto & T_{p_{i}}\Gamma
\end{array}
\end{equation}
Furthermore, $f_{i}$ is birational and defines an isomorphism on $\mathcal{M}_{0,n}$. The birational maps $f_j\circ f_i^{-1}$
  \[
  \begin{tikzpicture}[xscale=1.5,yscale=-1.2]
    \node (A0_1) at (1, 0) {$\cM_{0,n}$};
    \node (A1_0) at (0, 1) {$\mathbb{P}^{n-3}$};
    \node (A1_2) at (2, 1) {$\mathbb{P}^{n-3}$};
    \path (A1_0) edge [->,dashed]node [auto] {$\scriptstyle{f_{j}\circ f_{i}^{-1}}$} (A1_2);
    \path (A0_1) edge [->]node [auto] {$\scriptstyle{f_{j}}$} (A1_2);
    \path (A0_1) edge [->]node [auto,swap] {$\scriptstyle{f_{i}}$} (A1_0);
  \end{tikzpicture}
  \]
are standard Cremona transformations of $\mathbb{P}^{n-3}$ \cite[Proposition 2.12]{Ka}. For any $i = 1,\dots,n$ the class $\Psi_{i}$ is the line bundle on $\cM_{0,n}$ whose fiber on $[C,x_{1},\dots,x_{n}]$ is the tangent line $T_{p_{i}}C$. From the previous description we see that the line bundle $\Psi_{i}$ induces the birational morphism $f_{i}:\cM_{0,n}\rightarrow\mathbb{P}^{n-3}$, that is $\Psi_{i} = f_{i}^{*}\mathcal{O}_{\mathbb{P}^{n-3}}(1)$. In \cite{Ka} Kapranov proved that $\Psi_{i}$ is big and globally generated, and that the birational morphism $f_{i}$ is an iterated blow-up of the projections from $p_{i}$ of the points $p_{1},\dots,\hat{p_{i}},\dots,p_{n}$ and of all strict transforms of the linear spaces that they generate, in order of increasing dimension. 

\begin{Construction}\label{kblusym}
Fix $(n-1)$-points $p_{1},\dots,p_{n-1}\in\mathbb{P}^{n-3}$ in linear general position:
\begin{itemize}
\item[(1)] Blow-up the points $p_{1},\dots,p_{n-1}$, 
\item[(2)] Blow-up the strict transforms of the lines $\Span{p_{i_{1}},p_{i_{2}}}$, $i_{1},i_{2} = 1,\dots,n-1$,\\
\vdots
\item[($k$)] Blow-up the strict transforms of the $(k-1)$-planes $\Span{p_{i_{1}},\dots,p_{i_{k}}}$, $i_{1},\dots,i_{k} = 1,\dots,n-1$,\\
\vdots
\item[($n-4$)] Blow-up the strict transforms of the $(n-5)$-planes $\Span{p_{i_{1}},\dots,p_{i_{n-4}}}$, $i_{1},\dots,i_{n-4} = 1,\dots,n-1$.
\end{itemize}

Now, consider the moduli spaces of weighted pointed curves $X_{k}[n]:=\cM_{0,A[n]}$ for $k = 1,\dots,n-4$, such that
\begin{itemize}
\item[-] $a_{i}+a_{n}>1$ for $i=1,\dots,n-1$,
\item[-] $a_{i_{1}}+\dots+a_{i_{r}}\leq 1$ for each $\{i_{1},\dots,i_{r}\}\subset\{1,\dots,n-1\}$ with $r\leq n-k-2$,
\item[-] $a_{i_{1}}+\dots+a_{i_{r}}> 1$ for each $\{i_{1},\dots,i_{r}\}\subset\{1,\dots,n-1\}$ with $r> n-k-2$.
\end{itemize}
The composition of these blow-up morphism here above is the morphism $f_{n}:\cM_{0,n}\rightarrow\mathbb{P}^{n-3}$ induced by the psi-class $\Psi_{n}$. Identifying $\cM_{0,n}$ with $V(p_{1},\dots,p_{n})$, and fixing a general $(n-3)$-plane $H\subset\mathbb{P}^{n-2}$, the morphism $f_{n}$ associates to a curve $C\in V(p_{1},\dots,p_{n})$ the point $T_{p_{n}}C\cap H$.
\end{Construction}
In \cite[Section 2.1.2]{Ha} Hassett considers a natural variation of the moduli problem of weighted pointed rational stable curves by considering weights of the type $\widetilde{A}[n]= (a_1,\dots,a_n)$ such that $a_{i}\in\mathbb{Q}$, $0< a_i\leq 1$ for any $i = 1,\dots,n$, and $\sum_{i=1}^{n}a_{i} =2$.

By \cite[Section 2.1.2]{Ha} we may construct an explicit family of such weighted curves $\mathcal{C}(\widetilde{A})\rightarrow\cM_{0,n}$ over $\cM_{0,n}$ as an explicit blow-down of the universal curve over $\cM_{0,n}$.\\
Furthermore, if $a_i<1$ for any $i=1,\dots,n$ we may interpret the geometric invariant theory quotient $(\mathbb{P}^1)^n\quot PGL(2)$ with respect to the linearization $\mathcal{O}(a_1,\dots,a_n)$ as the moduli space $\cM_{0,\widetilde{A}[n]}$ associated to the family $\mathcal{C}(\widetilde{A})$.

\begin{Remark}\label{GITHas}
Note that we may interpret the GIT quotient $(\mathbb{P}^1)^n(b)\quot PGL(2)$ as a moduli space $\cM_{0,\widetilde{A}[n]}$ by taking the weights $a_i = \frac{2}{|b|}b_i$. Conversely, given the space $\cM_{0,\widetilde{A}[n]}$ with $(a_1,\dots,a_n) = (\frac{\alpha_1}{\beta_1},\dots,\frac{\alpha_n}{\beta_n})$ such that $\sum_{i=1}^na_i=2$ we may consider the GIT quotient $(\mathbb{P}^1)^n(b)\quot PGL(2)$ with $b_i = a_iM$, where $M= LCM(\beta_i)$.
\end{Remark}

\begin{Remark}\label{redwst}
Let $\cM_{0,\widetilde{A}[n]}$ be a moduli space with weights $\widetilde{a}_i$ summing up to two, and let $\cM_{0,A[n]}$ be a moduli  space with weights $a_i\geq \widetilde{a}_i$ for any $i=1,\dots,n$. By \cite[Theorem 8.3]{Ha} there exists a reduction morphism $\rho_{\widetilde{A}[n],A[n]}:\cM_{0,A[n]}\rightarrow\cM_{0,\widetilde{A}[n]}$ operating as the standard reduction morphisms in \ref{reduction}.
\end{Remark}

\begin{Proposition}\label{res}
Let $\phi_{\mathcal{L}}:\mathbb{P}^{n-3}\dasharrow (\mathbb{P}^1)^n(b)\quot PGL(2)\subset\mathbb{P}(H^0(\mathbb{P}^{n-3},\mathcal{L})^{*})$ be the rational map in Theorem \ref{Kum}, and let $f_i:\cM_{0,n}\rightarrow\mathbb{P}^{n-3}$ be the morphism in (\ref{KapMor}). Then there exists a reduction morphism $\rho:\cM_{0,n}\rightarrow (\mathbb{P}^1)^n(b)\quot PGL(2)$ making the following diagram 
 \[
  \begin{tikzpicture}[xscale=3.2,yscale=-1.5]
    \node (A0_0) at (0, 0) {$\cM_{0,n}$};
    \node (A1_0) at (0, 1) {$\mathbb{P}^{n-3}$};
    \node (A1_1) at (1, 1) {$(\mathbb{P}^1)^n(b)\quot PGL(2)$};
    \path (A0_0) edge [->,swap]node [auto] {$\scriptstyle{f_n}$} (A1_0);
    \path (A1_0) edge [->,dashed]node [auto] {$\scriptstyle{\phi_{\mathcal{L}}}$} (A1_1);
    \path (A0_0) edge [->]node [auto] {$\scriptstyle{\rho}$} (A1_1);
  \end{tikzpicture}
  \]
commutative.  
\end{Proposition}

\begin{proof}
As observed in \cite{Ku03} via the theory of associated points, each point $x\in \mathbb{P}^{n-3}$ which is linearly general with respect to the $n-1$ fixed points in $\mathbb{P}^{n-3}$ defines a configuration of points on the unique rational normal curve of degree $n-3$ passing through the $(n-1)+1=n$ points. Moreover, this configuration is the image of $x$ in $(\mathbb{P}^1)^n(b)\quot PGL(2)$ via $\phi_{\mathcal{L}}$. Via the identification $V_0(p_1,\dots,p_n)\cong \mathcal{M}_{0,n}$ in Theorem \ref{kaphilb}, one easily obtains the claim.
\end{proof}
From the next section on, we will always omit the vector $b$ of the polarization since it will always be $(1, \dots, 1)$.

\section{Birational geometry of GIT quotients of $(\mathbb{P}^1)^n$}\label{secbir}
In this section we will study some birational aspects of the geometry of the GIT quotients we introduced. In particular, we will describe their Mori cone and show that its extremal rays are generated by 1-dimensional strata of the boundary. As usual, we will denote by $\Sigma_m$ the GIT quotient $(\P^1)^{m+3}\quot PGL(2)$ with respect to the symmetric polarization.

\stepcounter{thm}
\subsection{The action of the standard Cremona transformation}
Let $p_1,\dots,p_{n+1}\in\mathbb{P}^n$ be general points, and $X_{n+1}^n$ be the blow-up of $\mathbb{P}^n$ at $p_1,\dots,p_{n+1}$. We may assume that $p_1 = [1:0:\dots:0],\dots,p_{n+1} = [0:\dots:0:1]$. Let us consider the standard Cremona transformation:
$$
\begin{array}{cccc}
\psi_n: & \mathbb{P}^n & \dasharrow & \mathbb{P}^n\\
 & \left[x_0:\dots:x_n\right] & \longmapsto & [\frac{1}{x_0}:\dots:\frac{1}{x_n}]
\end{array}
$$
Note that $\psi_n\circ \psi_n = Id_{\mathbb{P}^n}$, and $\psi^{-1}_n = \psi_n$. Let $H_1,\dots,H_{n+1}$ be the coordinate hyperplanes of $\mathbb{P}^n$. Then $\psi_n$ is not defined on the locus $\bigcup_{1\leq i< j\leq n+1}H_i\cap H_j$. Furthermore, $\psi_n$ is an isomorphism off of the union $\bigcup_{1\leq i\leq n+1}H_i$.

Now, $\psi_n$ induces a birational transformation $\widetilde{\psi}_n:X_{n+1}^n\dasharrow X_{n+1}^n$. Note that, since $\psi_n$ contracts the hyperplane $H_i$ spanned by the $n$ points $p_1,\dots,\hat{p}_i,\dots,p_{n+1}$ onto the point $p_i$, the map $\widetilde{\psi}_n$ maps the strict transform of $H_i$ onto the exceptional divisor $E_i$. Therefore $\widetilde{\psi}_n$ is an isomorphism in codimension one. Indeed, it is a composition of flops. In particular $\widetilde{\psi}_n$ induces an isomorphism $\Pic(X_{n+1}^n)\rightarrow\Pic(X_{n+1}^n)$.

\begin{say}
Let us define $\mathcal{L}_{2g-1}: = \mathcal{L}_{2g-1,g}(g-1,\dots,g-1)$ as the linear system of degree $g$ forms on $\mathbb{P}^{2g-1}$ vanishing with multiplicity $g-1$ at $2g+1$ general points $p_{1},\dots,p_{2g+1}\in\mathbb{P}^{2g-1}$. Note that in this case $n = 2g+2$ and $b = (1,\dots,1)$. In \cite[Theorem 4.1]{Ku00} Kumar proved that $\mathcal{L}_{2g-1}$ induces a birational map
\stepcounter{thm}
\begin{equation}\label{mapseg}
\sigma_g:\mathbb{P}^{2g-1}\dasharrow\mathbb{P}(H^0(\mathbb{P}^{2g-1},\mathcal{L}_{2g-1})^{*}) = \mathbb{P}^N
\end{equation}
and that the GIT quotient $\Sigma_{2g-1}$, that is the Segre $g$-variety, in Kumar's paper \cite{Ku00}, is obtained as the closure of the image of $\sigma_g$ in $\mathbb{P}(H^0(\mathbb{P}^{2g-1},\mathcal{L}_{2g-1})^{*})$.

Furthermore when $n = 2g+1$ and $b = (1,\dots,1)$, by Theorem \ref{Kum}, $\Sigma_{2g}\subset \mathbb{P}(H^0(\mathbb{P}^{2g},\mathcal{L}_{2g})^*)$ 
is the closure of the image of the rational map induced by the linear 
system $\mathcal{L}_{2g}$, where we define $\mathcal{L}_{2g}$ as given by degree $2g+1$ hypersurfaces in $\mathbb{P}^{2g}$ with multiplicity $2g-1$ at $p_i$ for $i = 1,\dots,2g+2$. We will denote by 
\stepcounter{thm}
\begin{equation}\label{mapeven}
\mu_g:\mathbb{P}^{2g}\dasharrow \Sigma_{2g}\subset \mathbb{P}(H^0(\mathbb{P}^{2g},\mathcal{L}_{2g})^{*}) = \mathbb{P}^N
\end{equation}
this rational map. 
\end{say}

\begin{Lemma}\label{founCrem}
Let $D\subset\mathbb{P}^n$ be a hypersurface of degree $d$ having points of multiplicities $m_1,\dots,m_{n+1}$ in $p_1,\dots,p_{n+1}$, and let $\psi_n:\mathbb{P}^{n}\dasharrow\mathbb{P}^n$ be the standard Cremona transformation of $\mathbb{P}^n$. Then 
$$\deg(\psi_n(D)) = dn-\sum_{i=1}^{n+1}m_i \quad and \quad \mult_{p_i}\psi_n(D) = d(n-1)-\sum_{j\neq i}m_j$$
for any $i = 1,\dots,n+1$.
\end{Lemma}
\begin{proof}
Let $X_{n+1}^n = Bl_{p_1,\dots,p_{n+1}}\mathbb{P}^n$, and $\widetilde{\psi}_n:X_{n+1}^n\dasharrow X_{n+1}^n$ be the birational map induced by $\psi_n$. The strict transform of $D$ in $X_{n+1}^n\dasharrow X_{n+1}^n$ can be written as $\widetilde{D} \cong dH-\sum_{i=1}^{n+1}m_iE_i$.\\
Now, since $\widetilde{\psi}_{n*}H = nH-\sum_{i=1}^{n+1}(n-1)E_i$, and $\widetilde{\psi}_{n*}E_i = H-\sum_{j\neq i}E_i$ we get the formula
$$
\begin{array}{ll}
\widetilde{\psi}_{n*}D = & d(nH-\sum_{i=1}^{n+1}E_i)-\sum_{i=1}^{n+1}m_i(H-\sum_{j\neq i}E_j)=\\
 & dnH-d\sum_{i=1}^{n+1}(n-1)E_i-\sum_{i=1}^{n+1}H+\sum_{i=1}^{n+1}m_i\sum_{j\neq i}E_j =\\ 
 & (dn-\sum_{i=1}^{n+1}m_i)H-\sum_{i=1}^{n+1}(d(n-1)-\sum_{j\neq i}m_j)E_j
\end{array} 
$$
which gives exactly the statement.
\end{proof}

\begin{Proposition}\label{autcrem}
The standard Cremona transformation $\psi_{m}:\mathbb{P}^{m}\dasharrow\mathbb{P}^{m}$ induces an automorphism of the GIT quotient $\Sigma_{m}$.
\end{Proposition}
\begin{proof}
Let us consider the case $m = 2g-1$. The proof in the even dimensional case will be analogous. Let $D\in H^0(\mathbb{P}^{2g-1},\mathcal{L}_{2g-1})$ be a section of the linear system inducing the map (\ref{mapseg}). By Lemma \ref{founCrem} we get $\deg(\psi_{2g-1}(D)) = g(2g-1)-\sum_{i=1}^{2g}m_i = g(2g-1)-2g(g-1)=g$
and $\mult_{p_i}\psi_{2g-1}(D) = g(2g-2)-\sum_{j\neq i}m_j = g(2g-2)-(2g-1)(g-1)=g-1$
for $i=1,\dots,2g$. Furthermore, since $\psi_{2g-1}$ is an isomorphism in a neighborhood of $p_{2g+1}$ we have that $\mult_{p_{2g+1}}\psi_{2g-1}(D)=g-1$ as well.

Therefore, $\psi_{2g-1}$ acts on the sections of $\mathcal{L}_{2g-1}$, and hence it induces an automorphism $\widetilde{\psi}_{2g-1}$ of $\mathbb{P}(H^0(\mathbb{P}^{2g-1},\mathcal{L}_{2g-1})^{*})$ that keeps $\Sigma_{2g-1}$ stable.
\end{proof}

\begin{Remark}\label{cremodd}
Note that the automorphism induced by the Cremona and the group $S_{m+2}$ permuting the points $p_1,\dots,p_{m+2}\in\mathbb{P}^{m}$ generates the symmetric group $S_{m+3}$ acting on $\Sigma_{m}$ by permuting the marked points.
\end{Remark}

\stepcounter{thm}
\subsection{The odd dimensional case}\label{oddismi}
We start by studying the case where the GIT quotients parametrize an even number of points, that is $\Sigma_{2g-1}$, for $g\geq 2$. 

\begin{Proposition}\label{resseg}
Let $p_1,\dots,p_{2g+1}\in\mathbb{P}^{2g-1}$ be points in general position, and let $X^{2g-1}_{g-2}$ be the variety obtained at the step $g$ of Construction \ref{kblusym}. Then we have the following commutative diagram 
\[
  \begin{tikzpicture}[xscale=3.2,yscale=-1.5]
    \node (A0_0) at (0, 0) {$X^{2g-1}_{g-2}$};
    \node (A1_0) at (0, 1) {$\mathbb{P}^{2g-1}$};
    \node (A1_1) at (1, 1) {$\Sigma_{2g-1}\subset\mathbb{P}^N.$};
    \path (A0_0) edge [->,swap]node [auto] {$\scriptstyle{f}$} (A1_0);
    \path (A1_0) edge [->,dashed]node [auto] {$\scriptstyle{\sigma_g}$} (A1_1);
    \path (A0_0) edge [->]node [auto] {$\scriptstyle{\widetilde{\sigma}_g}$} (A1_1);
  \end{tikzpicture}
  \]
That is, the blow-up morphism $f:X^{2g-1}_{g-2}\rightarrow\mathbb{P}^{2g-1}$ resolves the rational map $\sigma_g$. In particular, $\Sigma_{2g-1}$ has exactly $\binom{2g-1}{g}$ singular points. 
\end{Proposition}
\begin{proof}
By Construction \ref{kblusym} the blow-up $X^{2g-1}_{g-2}$ may be interpreted as the moduli space $\cM_{0,A[2g+2]}$ with $A[2g+2] = \left(\frac{1}{g},\dots,\frac{1}{g},1\right)$. 

Furthermore, by Remark \ref{GITHas} $\Sigma_{2g-1}$ is the singular moduli space $\cM_{0,\widetilde{A}[2g+2]}$ with weights $A[2g+2] = \left(\frac{1}{g+1},\dots,\frac{1}{g+1}\right)$, and by Proposition \ref{res} the morphism $\widetilde{\sigma}_g:X^{2g-1}_{g-2}\rightarrow \Sigma_g$ is exactly the reduction morphism $\rho_{\widetilde{A}[2g+2],A[2g+2]}:\cM_{0,A[2g+2]}\rightarrow \cM_{0,\widetilde{A}[2g+2]}$ defined just by lowering the weights. 

Finally, note that $\widetilde{\sigma}_g$ contracts the $(g-1)$-planes through the $p_i$ and is an isomorphism elsewhere. So $\Sigma_{2g-1}$ has exactly $\binom{2g-1}{g}$ singular points coming from the contraction of these $(g-1)$-planes.  
\end{proof}

This fairly simple result has some interesting consequences, that we will illustrate in the rest of this section. Anyway, first we need a technical lemma.

\begin{Lemma}\label{teclemma}
Let $X$ be a normal projective variety and $f:X\dasharrow Y$ a birational map of projective varieties not contracting any divisor. Then $K_X \sim f^{-1}_{*}K_Y$.
\end{Lemma}
\begin{proof}
Since $X$ is normal $f$ is defined in codimension one. Let $U\subseteq X$ be a dense open subset whose complementary set has codimension at least two where $f$ is defined. Since $f$ does not contract any divisor then its exceptional set has at least codimension two, hence we may assume that $f_{|U}$ is an isomorphism onto its image. Therefore $K_{X|U} \sim (f^{-1}_{*}K_Y)_{|U}$, and since $U$ is at least of codimension two we get the statement. 
\end{proof}

The first consequence of Proposition \ref{resseg} is the following.

\begin{Lemma}\label{can}
The canonical sheaf of $\Sigma_{2g-1}\subset\mathbb{P}^N$ is $K_{\Sigma_{2g-1}}\cong \mathcal{O}_{\Sigma_{2g-1}}(-2)$.
\end{Lemma}
\begin{proof}
Let $X^{2g-1}_1$ be the blow-up of $\mathbb{P}^{2g-1}$ at $2g+1$ general points $p_1,\dots,p_{2g+1}$, and let $f_g:X^{2g-1}_1\dasharrow \Sigma_{2g-1}$ be the birational map induced by the morphism $\widetilde{\sigma}_g:X^{2g-1}_{g-2}\rightarrow\Sigma_{2g-1}$ in Proposition \ref{resseg}. We will denote by $H$ the pull-back in $X_1^{2g-1}$ of the hyperplane section of $\mathbb{P}^{2g-1}$, and by $E_{1},\dots,E_{2g+1}$ the exceptional divisors.  

The interpretation of $\widetilde{\sigma}_g$ as a reduction morphism in the proof of Proposition \ref{resseg} yields that $f_g$ does not contract any divisor. Indeed $f_g$ contracts just the strict transforms of the $(g-1)$-planes generated by $g$ of the blown-up points. Since $\sigma_g$ is induced by the linear system $\mathcal{L}_{2g-1}$, the pull-back via $f_g$ of a hyperplane section of $\Sigma_{2g-1}$ is the strict transform of a hypersurface of degree $g$ in $\mathbb{P}^{2g-1}$ having multiplicity $g-1$ at $p_1,\dots,p_{2g+1}$. Since
$$-K_{X^{2g-1}_1}\sim 2gH-(2g-2)\sum_{i=1}^{2g+1}E_i = 2\left(gH-(g-1)\sum_{i=1}^{2g+1}E_i\right)$$ 
we have that $-K_{X^{2g-1}_1}\sim f_g^{*}\mathcal{O}_{\Sigma_{2g-1}}(2)$. Now, in order to conclude it is enough to apply Lemma \ref{teclemma} to the birational map $f_g:X^{2g-1}_1\dasharrow \Sigma_{2g-1}$.
\end{proof}

\begin{Proposition}\label{picsegre}
The divisor class group of $\Sigma_{2g-1}$ is $\Cl(\Sigma_{2g-1})\cong \mathbb{Z}^{2g+2}$. Furthermore $\Pic(\Sigma_{2g-1})$ is torsion free.
\end{Proposition}
\begin{proof}
The classical case of the Segre cubic $\Sigma_3\subset\mathbb{P}^4$ has been treated in \cite[Section 3.2.2]{Hu96}. Hence we may assume that $g\geq 2$. 

Let $Y = H\cap\Sigma_{2g-1}$ be a general hyperplane section of $\Sigma_{2g-1}$. Since $\dim(\Sing(\Sigma_{2g-1}))=0$ by Bertini's theorem, see \cite[Corollary 10.9]{Ha77} and \cite[Remark 10.9.2]{Ha77}, $Y$ is smooth. Note that $X = (\widetilde{\sigma}_g)^{-1}(Y)\subset X^{2g-1}_{g-2}$ is the strict transform via $f$ of a general element of the linear system $\mathcal{L}_{2g-1}$ inducing $\sigma_g$. Therefore, $X$ is smooth and $\widetilde{\sigma}_{g|X}:X\rightarrow Y$ is a divisorial contraction between smooth varieties.

Since $\dim(X)>2$, the Grothendieck-Lefschetz theorem, see for instance \cite[Theorem A]{Ba78}, yields that the natural restriction morphism $\Pic(X^{2g-1}_{g-2})\rightarrow\Pic(X)$ is an isomorphism. Therefore $\Pic(X)\cong \Pic(X^{2g-1}_{g-2})\cong \mathbb{Z}^h$, with $h  = 1+(2g+1)+\binom{2g+1}{2}+\dots+\binom{2g+1}{g}$. By the interpretation of $\widetilde{\sigma}_{g}$ as a reduction morphism in Proposition \ref{resseg} we know that the codimension one part of the exceptional locus of $\widetilde{\sigma}_{g}$ consists of the exceptional divisors the blown-up positive dimensional linear subspaces of $\mathbb{P}^{2g-1}$. Therefore, $\Pic(Y)\cong \mathbb{Z}^{2g+2}$. Indeed $\Pic(Y)$ is generated by the images via $\widetilde{\sigma}_{g|X}$ of the pull-back of the hyperplane section of $\mathbb{P}^{2g-1}$ and of the exceptional divisors over the $2g+1$ blown-up points. However, since $\Sigma_{2g-1}$ is not smooth we can not conclude by the Grothendieck-Lefschetz theorem that its Picard group is isomorphic to $\mathbb{Z}^{2g+2}$ as well.

On the other hand, thanks to a version for normal varieties of the Grothendieck-Lefschetz theorem \cite[Theorem 1]{RS06}, we get that $\Cl(\Sigma_{2g-1})\cong\Cl(Y)$, and since $Y$ is smooth we have $\Cl(Y)\cong\Pic(Y)$. 

Finally, by \cite[Corollary 2, Page 305]{Kl66} we have that, even when the ambient variety is singular, the restriction morphism in the Grothendieck-Lefschetz theorem is injective. Therefore, we have an injective morphism $\Pic(\Sigma_{2g-1})\hookrightarrow \Pic(Y)\cong\mathbb{Z}^{2g+2}$, and hence $\Pic(\Sigma_{2g-1})$ is torsion free. 
\end{proof}

This allows us to compute the Picard group of $\Sigma_{2g-1}$. By different methods, this computation was also carried out in \cite{MS16}. Note that $\Sigma_{2g-1}$ is not $\mathbb{Q}$-factorial, and many boundary divisors are just Weil divisors. This is the reason why the rank of $\Pic(\Sigma_{2g-1})$ drops dramatically from the rank of $\Cl(\Sigma_{2g-1})$.

\begin{Proposition}\label{picpari}
The Picard group of the GIT quotient $\Sigma_{2g-1}\subset\mathbb{P}^N$ is $\Pic(\Sigma_{2g-1})\cong \mathbb{Z}\left\langle H\right\rangle$, where $H$ is the hyperplane class. In particular, we have that $\Nef(\Sigma_{2g-1})\cong\mathbb{R}^{\geq 0}$.
\end{Proposition}
\begin{proof}
Let $\widetilde{\sigma}_g: X^{2g-1}_{g-2}\rightarrow \Sigma_{2g-1}\subset\mathbb{P}^N$ be the resolution of the birational map $\sigma_g:\mathbb{P}^{2g+1}\dasharrow \Sigma_{2g-1}$ in Proposition \ref{resseg}. Note that $\Pic(X^{2g-1}_{g-2})\cong \mathbb{Z}^{\rho(X^{2g-1}_{g-2})}$, where $\rho(X^{2g-1}_{g-2}) = 1+(2g+1)+\binom{2g+1}{2}+\dots+\binom{2g+1}{g}$, and that $\widetilde{\sigma}_g$ contracts all the exceptional divisors of the blow-up $f:X^{2g-1}_{g-2}\rightarrow\mathbb{P}^{2g-1}$ over positive dimensional linear subspaces. Now, let $E_i\subset X_1^{2g-1}$ be the exceptional divisor over the point $p_i\in \mathbb{P}^{2g-1}$, where $X^{2g-1}_{1}$ is the blow-up of $\mathbb{P}^{2g-1}$ at $p_1,\dots,p_{2g+1}$, and let $\widetilde{E}_i$ be its strict transform in $X^{2g-1}_{g-2}$. Furthermore, denote by $e_i$ the class of a general line in $E_i\subset X_{0}^{2g-1}$. 

By Proposition \ref{resseg} we know that, besides the divisors over the positive dimensional linear subspaces, $\widetilde{\sigma}_g$ also contracts the strict transforms $S_1,\dots,S_r$, with $r = \binom{2g+1}{g}$, in $X^{2g-1}_{g-2}$ of the $(g-1)$-planes $\left\langle p_{i_1},\dots,p_{i_g}\right\rangle$. Note that the contraction of $S_i$ is given by the contraction of the strict transforms of the degree $g-1$ rational normal curves in $\left\langle p_{i_1},\dots,p_{i_g}\right\rangle$ passing through $p_{i_1},\dots,p_{i_g}$. In fact, any degree $g-1$ rational normal curve through $p_{i_1},\dots,p_{i_g}$ is contained in $\left\langle p_{i_1},\dots,p_{i_g}\right\rangle$ and a morphism contracts $S_i$ to a point if and only if it contracts all these curves to a point. We may write the class in the cone $N_1(X^{2g-1}_{g-2})_{\mathbb{R}}\cong \mathbb{R}^{\rho(X^{2g-1}_{g-2})}$ of the strict transform of such a rational normal curve as
\stepcounter{thm}
\begin{equation}\label{classrnc}
(g-1)l-e_{i_1}-\dots-e_{i_g} 
\end{equation}
where $l$ is the pull-back of a general line in $\mathbb{P}^{2g-1}$. Note that the classes in (\ref{classrnc}) generate the hyperplane
$$\left\lbrace gl+\sum_{j=1}^{2g+1}(g-1)e_i = 0\right\rbrace$$
in $N_1(X^{2g-1}_{1})_{\mathbb{R}}\cong \mathbb{R}^{2g+2}$. Therefore, the birational morphism $\widetilde{\sigma}_g: X^{2g-1}_{g-2}\rightarrow \Sigma_{2g-1}$ contracts the locus spanned by classes of curves generating a subspace of dimension $\binom{2g+1}{2}+\dots+\binom{2g+1}{g} + 2g+1$ of $N_1(X^{2g-1}_{g-2})_{\mathbb{R}}\cong \mathbb{R}^{\rho(X^{2g-1}_{g-2})}$, and then 
$$\rho(\Sigma_{2g-1}) = \rho(X^{2g-1}_{g-2})-\left(\binom{2g+1}{2}+\dots+\binom{2g+1}{g} + 2g+1\right) = 1$$
Since by Proposition \ref{picsegre} the Picard group of $\Sigma_{2g-1}$ is torsion free and $\Sigma_{2g-1}\subset\mathbb{P}^N$ contains lines we conclude that $\Pic(\Sigma_{2g-1})=\mathbb{Z}\left\langle H\right\rangle$, where $H$ is the hyperplane class.
\end{proof}

\stepcounter{thm}
\subsection{Linear subspaces of dimension $g$ in $\Sigma_{2g-1}$}\label{gSeg}
In this section we will study a particular configuration of $g$-planes contained in $\Sigma_{2g-1}$, and then we will exploit this configuration to compute the symmetries of $\Sigma_{2g-1}$.

\begin{say}\label{gplanes}
Let $H_{I} = H_{i_1,\dots,i_{g+1}}$ be the $g$-plane in $\mathbb{P}^{2g-1}$, linear span of the points $p_{i_1},\dots,p_{i_{g+1}}$, and let $\mathcal{L}_{2g-1|H_{I}}$ be the restriction to $H_I$ of the linear system $\mathcal{L}_{2g-1}$ inducing $\sigma_g$. Then $\mathcal{L}_{2g-1|H_I}$ is the linear system of degree $g$ hypersurfaces in $H_I\cong\mathbb{P}^g$ having multiplicity $g-1$ at $p_{i_1},\dots,p_{i_{g+1}}$. This means that $\sigma_{g|H_I}$ is the standard Cremona transformation of $\mathbb{P}^g$. Therefore, $\sigma_g(H_I)$ is a $g$-plane in $\Sigma_{2g-1}$ passing through the singular points given by the contractions of the $(g-1)$-planes generated by subsets of cardinality $g$ of $\{p_{i_1},\dots,p_{i_{g+1}}\}$. Now let $\Pi_{I^c}$ be the $(g-1)$-plane generated by the points in $\{p_1,\dots,p_{2g+1}\}\setminus\{p_{i_1},\dots,p_{i_{g+1}}\}$. Note that $H_{I^c}$ intersects $H_{I}$ in one point, hence $\sigma_g(H_I)$ passes through the singular point $\sigma_g(H_{I^c})$ as well. We conclude that there are $g+1+1 = g+2$ singular points of $\Sigma_{2g-1}$ lying on the $g$-plane $\sigma_g(H_I)$.

Now, let $E^{g-2}_I\subset X^{2g-1}_{g-2}$ be the exceptional divisor over the $(g-2)$-plane $H^{g-2}_I$ generated by the $p_i$ for $i\in I$. Then $E^{g-2}_I$ is a $\mathbb{P}^g$-bundle over the strict transform of $H^{g-2}_I$ in $X^{2g-1}_{g-3}$. For any $j\notin I$ let $H^{g-1}_{I\cup\{j\}}$ be the $(g-1)$-plane generated by the $p_i$ for $i\in I$ and $p_j$. Note that the strict transform of $H^{g-1}_{I\cup\{j\}}$ intersects $E^{g-2}_I$ along a section $s$ which is mapped by the blow-up morphism isomorphically onto $H^{g-2}_I$. Since the strict transform of $H^{g-1}_{I\cup\{j\}}$ is contracted to a point by $\tilde{\sigma}_g$, the section $s$ must be contracted to a point as well. Therefore, $\tilde{\sigma}_g(E_I^{g-1})$ is a $g$-plane passing through $g+2$ singular points of $\Sigma_{2g-1}$. 

So far we have found $\binom{2g+1}{g+1}+\binom{2g+1}{g-1}$ linear spaces of dimension $g$ in $\Sigma_{2g-1}$, and each of them contains at least $g+2$ of the $\binom{2g+1}{g}$ singular points of $\Sigma_{2g-1}$. We will divide the $g$-planes inside $\Sigma_{2g-1}$ passing through a singular point $p\in \Sigma_{2g-1}$ into two families according to their mutual intersection.
\end{say}

\begin{say}\label{families}
The singular locus of $\Sigma_{2g-1}$ consists of $\binom{2g+1}{g}$ points corresponding to the $\binom{2g+1}{g}$ linear subspaces $\left\langle p_{i_1},\dots,p_{i_g}\right\rangle$ with $\{i_1,\dots,i_g\}\subset\{1,\dots,2g+1\}$ contracted by $\sigma_{g}$.

Now, let $p\in \Sigma_{2g-1}$ be a singular point. So far, we found $2g+2$ linear subspaces of dimension $g$ in $\Sigma_{2g-1}$ passing through $p\in \Sigma_{2g-1}$. These $g$-planes may be divided in two families: 
$$\textbf{A}_p=\{\alpha_1,\dots,\alpha_{g+1}\}, \: \textbf{B}_p=\{\beta_1,\dots,\beta_{g+1}\}$$
with the following properties:
\begin{itemize}
\item[-] $\alpha_i\cap\alpha_j = \beta_i\cap\beta_j = \{p\}$ for any $i,j=1,\dots,g+1$,
\item[-] $\alpha_i\cap\beta_j =\left\langle p,q_{ij}\right\rangle$ for any $i,j=1,\dots,g+1$, where $q_{ij}\in \Sigma_{2g-1} $ is a singular point $q_{ij}\neq p$.
\end{itemize}
Note that $p = \widetilde{\sigma}_g(H_I)$ with $I = \{i_1,\dots,i_g\}$, $\textbf{A}_p$ is the set of $\widetilde{\sigma}_g(E_J^{g-1})$ with $J\subset I$ and $|J| = g-1$ plus $\widetilde{\sigma}_g(H_{I^c})$, $\textbf{B}_p$ is the set of $\sigma_g(H_{I\cup\{j\}})$ with $j\notin I$. The configuration is summarized in the following picture:
$$
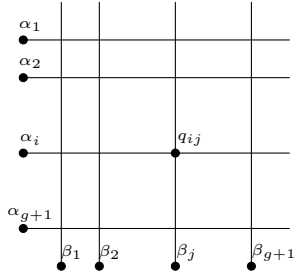
\begin{tikzpicture}[line cap=round,line join=round,>=triangle 45,x=0.5cm,y=0.5cm]
\clip(-3.5,-1.2) rectangle (4.1,6.1);
\draw (-2.,6.)-- (-2.,-1.);
\draw (-1.,6.)-- (-1.,-1.);
\draw (1.,6.)-- (1.,-1.);
\draw (3.,6.)-- (3.,-1.);
\draw (-3.,5.)-- (4.,5.);
\draw (-3.,4.)-- (4.,4.);
\draw (-3.,2.)-- (4.,2.);
\draw (-3.,0.)-- (4.,0.);
\begin{scriptsize}
\draw [fill=black] (-2.,-1.) circle (1.5pt);
\draw[color=black] (-1.73,-0.6291030050528003) node {$\scriptstyle{\beta_{1}}$};
\draw [fill=black] (-1.,-1.) circle (1.5pt);
\draw[color=black] (-0.73,-0.6291030050528003) node {$\scriptstyle{\beta_{2}}$};
\draw [fill=black] (1.,-1.) circle (1.5pt);
\draw[color=black] (1.3,-0.6291030050528003) node {$\scriptstyle{\beta_{j}}$};
\draw [fill=black] (3.,-1.) circle (1.5pt);
\draw[color=black] (3.6,-0.6291030050528003) node {$\scriptstyle{\beta_{g+1}}$};
\draw [fill=black] (-3.,5.) circle (1.5pt);
\draw[color=black] (-2.8094230554685238,5.374383650864099) node {$\scriptstyle{\alpha_{1}}$};
\draw [fill=black] (-3.,4.) circle (1.5pt);
\draw[color=black] (-2.8094230554685238,4.3780603334991675) node {$\scriptstyle{\alpha_{2}}$};
\draw [fill=black] (-3.,2.) circle (1.5pt);
\draw[color=black] (-2.8094230554685238,2.3598669470419966) node {$\scriptstyle{\alpha_{i}}$};
\draw [fill=black] (-3.,0.) circle (1.5pt);
\draw[color=black] (-2.8094230554685238,0.3672203123121321) node {$\scriptstyle{\alpha_{g+1}}$};
\draw [fill=black] (1.,2.) circle (1.5pt);
\draw[color=black] (1.4,2.3598669470419966) node {$\scriptstyle{q_{ij}}$};
\end{scriptsize}
\end{tikzpicture}
$$
where the black dots should all be interpreted as representing the singular point $p\in\Sigma_{2g-1}$.
\end{say}

\begin{say}\label{basis}
Let $R = \{I\subset \{1,\dots,2g\} \: |\: |I|=g\}$, $S = \{J\subset \{1,\dots,2g\} \: |\: |J|=g-2\}$, and $x_{I} = x_{i_1}\dots x_{i_g}$ where $I = \{i_1,\dots,i_g\}$, and the $x_i$ are homogeneous coordinates on $\mathbb{P}^{2g-1}$. By \cite[Theorem 4.1]{Ku00} we have that
\stepcounter{thm}
\begin{equation}\label{expkum}
H^{0}(\mathbb{P}^{2g-1},\mathcal{L}_{2g-1}) = \left\lbrace\sum_{I\in R}a_Ix_I \: |\: \sum_{J\subset I\in R}a_I=0 \: \forall \: J\in S\right\rbrace
\end{equation}
By (\ref{expkum}) we have $h^0(\mathbb{P}^{2g-1},\mathcal{L}_{2g-1}) = \binom{2g}{g}-\binom{2g}{g-2}$. Now, set $N= \binom{2g}{g}-\binom{2g}{g-2}-1$, and consider the expressions $s_i = \sum_{I\in R}a_I^i x_I$ for $i=0,\dots,N$. Let $H_1,\dots,H_{N+2}$ be $(g-1)$-planes in $\mathbb{P}^{2g-1}$ generated by subsets of cardinality $g$ of $\{[1:0:\dots:0],\dots,[0:\dots:0:1]\}\subset\mathbb{P}^{2g-1}$. Note that imposing $[s_1(H_i):\dots:s_N(H_i)] = [0:\dots:0:1:0:\dots:0]$, with the non-zero entry in the $i$-th position, for $i =1,\dots,N+1$ we get $N(N+1)$ equations. Furthermore, by setting $[s_1(H_{N+2}):\dots:s_N(H_{N+2})]$ equal to $[1:\dots:1]$ we get $N$ more equations. Recall that by (\ref{expkum}) for each $i=0,\dots,N$ there are $\binom{2g}{g-2}$ relations among the $a_I^i$. Therefore, we get $\binom{2g}{g-2}(N+1)$ more constraints. Summing up we have a linear system of $N(N+1)+N+\binom{2g}{g-2}(N+1)$ homogeneous equations in the $\binom{2g}{g}(N+1)$ indeterminates $a_{I}^{i}$. Note that
$$\binom{2g}{g}(N+1)-\left(N(N+1)+N+\binom{2g}{g-2}(N+1)\right) = 1$$ 
hence there exists a non-trivial solution. Let $s_0,\dots,s_N$ be the sections of $H^{0}(\mathbb{P}^{2g-1},\mathcal{L}_{2g-1})$ associated to such a solution. These sections yield an explicit realization of the map $\sigma_g:\mathbb{P}^{2g-1}\dasharrow \Sigma_{2g-1}\subset \mathbb{P}(H^{0}(\mathbb{P}^{2g-1},\mathcal{L}_{2g-1})^{*}) = \mathbb{P}^N$, $\sigma_g(x)= [s_0(x):\dots:s_N(x)]$. By construction and by the description of the singular locus of $\Sigma_{2g-1}$ in \ref{families}, we see that, with respect to this expression for $\sigma_g$, the points $[1:0:\dots:0],\dots,[0:\dots:0:1],[1:\dots:1]\in \mathbb{P}^N$ are singular points of $\Sigma_{2g-1}$. Hence, there are $N+2$ singular points of $\Sigma_{2g-1}$ that are in linear general position.
\end{say}

\stepcounter{thm}
\subsection{The even dimensional case}\label{evenio} 
In this section we investigate the geometry of the even counterpart of $\Sigma_{2g-1}$, that is the GIT quotient $(\mathbb{P}^1)^n(b)\quot PGL(2)$ with $b_i = 1$ for $i=1,\dots,n$ and $n$ odd. The quotient here has even dimension $2g$, for a positive integer $g$, and hence we will denote it by $\Sigma_{2g}$. We have $n=2g+3$. Note that in this case all the semistable points are indeed stable, and then $\Sigma_{2g}$ is smooth. 

\begin{Proposition}\label{resodd}
Let $X^{2g}_{g-1}$ be the variety obtained at the step $g$ of Construction \ref{kblusym}. Then there exists a morphism $\widetilde{\mu}_g:X^{2g}_{g-1}\rightarrow X(b)\quot PGL(2)$ making the following diagram  
\[
  \begin{tikzpicture}[xscale=3.2,yscale=-1.5]
    \node (A0_0) at (0, 0) {$X^{2g}_{g-1}$};
    \node (A1_0) at (0, 1) {$\mathbb{P}^{2g}$};
    \node (A1_1) at (1, 1) {$\Sigma_{2g}\subset\mathbb{P}^N$};
    \path (A0_0) edge [->,swap]node [auto] {$\scriptstyle{f}$} (A1_0);
    \path (A1_0) edge [->,dashed]node [auto] {$\scriptstyle{\mu_g}$} (A1_1);
    \path (A0_0) edge [->]node [auto] {$\scriptstyle{\widetilde{\mu}_g}$} (A1_1);
  \end{tikzpicture}
  \]
commute, where $f:X^{2g}_{g-1}\rightarrow\mathbb{P}^{2g}$ is the blow-up morphism and $\mu_g$ is the rational map in (\ref{mapeven}).
\end{Proposition}
\begin{proof}
By Construction \ref{kblusym} $X^{2g}_{g-1}\cong\cM_{0,A[2g+3]}$ with $A[2g+3] =  \left(\frac{2}{2g+2},\dots,\frac{2}{2g+2},1\right)$, and by Remark \ref{GITHas} $\Sigma_{2g}\cong \cM_{0,\widetilde{A}[2g+3]}$ with $\widetilde{A}[2g+3]=\left(\frac{2}{2g+3},\dots,\frac{2}{2g+3}\right)$. Therefore we may take $\widetilde{\mu}_g = \rho_{\widetilde{A}[2g+3],A[2g+3]}:\cM_{0,A[2g+3]}\rightarrow\cM_{0,\widetilde{A}[2g+3]}$ and argue as in the proof of Proposition \ref{resseg}.
\end{proof}

\stepcounter{thm}
\subsection{Linear subspaces of dimension $g$ in $\Sigma_{2g}\subset\mathbb{P}^N$}\label{linsubsodd} 
Let $H_I = H_{i_1,\dots,i_{g+1}}$ be the $g$-plane generated by $p_{i_1},\dots,p_{i_{g+1}}\in \mathbb{P}^{2g}$. The linear system $\mathcal{L}_{2g|H_I}$ is given by the hypersurfaces of degree $2g+1$ in $H_I\cong \mathbb{P}^g$ with multiplicity $2g-1$ at $p_{i_1},\dots,p_{i_{g+1}}$. Now, let $H_J\subset H_I$ be a $(g-1)$-plane generated by the points indexed by a subset $J\subset I$ with $|J| = g$. Then the general element $D$ of $\mathcal{L}_{2g|H_I}$ must contain $H_J$ with multiplicity $g(2g-1)-(g-1)(2g+1) = 1$. This means that the divisor $D$ equals $\bigcup_{\{J\subset I,\: | \: |J|=g\}}H_J$. Note that $\deg(D) = g+1$ and $\mult_{p_{i_j}}D = g$ for any $i_j\in I$. Therefore, $\mathcal{L}_{2g|H_I}$ is the linear system of hypersurfaces of degree $2g+1-(g+1)=g$ in $H_I\cong\mathbb{P}^g$ having multiplicity $2g-1-g = g-1$ at $p_{i_j}$ for any $i_j\in I$. This is the linear system of the standard Cremona transformation of $\mathbb{P}^g$. Therefore, $\mu_{g|H_I}(H_I)\subset \Sigma_{2g}\subset\mathbb{P}^N$ is a linear subspace of dimension $g$.

Now, let $E_I^{g-1}$ be the exceptional divisor over the strict transform of a $(g-1)$-plane of $\mathbb{P}^{2g}$ generated by $g$ of the $p_i$. Note that the reduction morphism $\widetilde{\mu}_g:X^{2g}_{g-1}\cong \cM_{0,A[2g+3]}\rightarrow \Sigma_{2g}\cong \cM_{0,\widetilde{A}[2g+3]}$ in Proposition \ref{resodd} contracts $E_I^{g-1}$ to a $g$-plane $\widetilde{\mu}_g(E_I^{g-1})\subset \Sigma_{2g}\subset\mathbb{P}^N$.

\begin{say}\label{confodd}
We found $\binom{2g+2}{g+1}+\binom{2g+2}{g}$ linear subspaces of dimension $g$ in $\Sigma_{2g}\subset\mathbb{P}^N$. We will denote by 
$$\textbf{C}=\{\gamma_1,\dots,\gamma_{c}\}, \: \textbf{D}=\{\delta_1,\dots,\delta_{d}\}$$
where $c = \binom{2g+2}{g+1}$ and $d = \binom{2g+2}{g}$, the families of the $g$-planes coming from the $H_I$ and the $E_I^{g-1}$ respectively. Note that:
\begin{itemize}
\item[-] on any $\delta_i$ we have $g+2$ distinguished points determined by the intersections with the $g-1$ of the $\gamma_j$ coming from $g$-planes in $\mathbb{P}^{2g}$ containing the $(g-1)$-plane associated to $\delta_i$,
\item[-] on any $\gamma_i$, say coming from $H_I$, we have $g+2$ distinguished points as well: $g+1$ of them coming from the exceptional divisors $E_J^{g-1}$ with $J\subset I$, $|J|=g$, and another one determined as the image of the point $H_I\cap H_{I^c}$, where $I_{c} =\{1,\dots,2g+2\}\setminus I$. 
\end{itemize}
Note that the $g+2$ distinguished points on the $\gamma_i$ and the $\delta_i$ are the same and come from intersecting $E_J^{g-1}$ with $H_I$ for $I = H\cup\{i\}$, $i\notin J$. These $g+2$ distinguished points are in linear general position. This is clear for the $\gamma_i$. In order to see that it is true for the $\delta_j$ as well, notice that we may map any $\delta_j$ to any $\gamma_i$ just by acting with a suitable permutation in $S_{2g+3}$ involving the standard Cremona transformation as in Remark \ref{cremodd}.

Finally note that on any $\gamma_i$ we have $g+2$ distinguished points and one of them is the intersection point of $\gamma_i$ with a $\gamma_j$. We call such a $\gamma_j$ the complementary of $\gamma_i$, and we denote it by $\gamma_j = \gamma_{i^c}$. Therefore $\gamma_i$ and $\gamma_{i^c}$ determine $2(g+2)-1$ distinguished points for any $i = 1,\dots,\frac{1}{2}\binom{2g+2}{g+1}$. Summing up we have 
\stepcounter{thm}
\begin{equation}\label{distpoints}
\frac{1}{2}\binom{2g+2}{g+1}(2(g+2)-1)= \frac{(2g+3)!}{2((g+1)!)^2}
\end{equation}
distinguished points in the configuration of $g$-planes $\textbf{C}\cup \textbf{D}$.
\end{say}

\begin{Remark}\label{genposodd}
Arguing as in \ref{basis} and using the description of the sections of $\mathcal{L}_{2g}$ in \cite[Section 3.3]{Ku03} we get that in $\Sigma_{2g}\subset \mathbb{P}(H^0(\mathbb{P}^{2g},\mathcal{L}_{2g})^{*}) = \mathbb{P}^N$ there are $N+2$ of the distinguished points described in \ref{confodd} that are in linear general position in $\mathbb{P}^N$.
\end{Remark}

\begin{say}\label{small}
Let us consider the following diagram:
\[
  \begin{tikzpicture}[xscale=1.5,yscale=-1.2]
    \node (A0_1) at (1, 0) {$X^{2g}_{g-1}$};
    \node (A1_0) at (0, 1) {$X^{2g}_1$};
    \node (A1_2) at (2, 1) {$\Sigma_{2g}$};
    \path (A0_1) edge [->] node [auto] {$\scriptstyle{\widetilde{\mu}_g}$} (A1_2);
    \path (A0_1) edge [->,swap] node [auto] {$\scriptstyle{h}$} (A1_0);
    \path (A1_0) edge [->,dashed] node [auto] {$\scriptstyle{\psi}$} (A1_2);
  \end{tikzpicture}
  \]
where $h:X^{2g}_{g-1}\rightarrow X^{2g}_1$ is the composition of blow-ups in Construction \ref{kblusym}. Note that, by interpreting the varieties appearing in the diagram as moduli spaces of weighted pointed curves, and $h$ and $\widetilde{\mu}_g$ as reduction morphisms as in the proof of Proposition \ref{resodd}, we see that the rational map $\psi:X^{2g}_1\dasharrow \Sigma_{2g}$ is a composition of flips of strict transforms of linear subspaces generated by subsets of $\{p_1,\dots,p_{2g+2}\}$ up to dimension $g-1$. In particular, $\psi$ is an isomorphism in codimension one and $\Pic(\Sigma_{2g})\cong \mathbb{Z}^{2g+3}$ is the free abelian group generated by the strict transforms via $\psi$ of $H,E_1,\dots,E_{2g+2}$. 
\end{say}

\begin{Lemma}\label{eff}
The effective cone $\Eff(X^{2g}_1)\subset \mathbb{R}^{2g+3}$ of $X^{2g}_1$ is the polyhedral cone generated by the classes of the exceptional divisors $E_i$ and of the strict transforms $H-\sum_{i\in I}E_i$, $I\subset\{1,\dots,2g+2\}$, with $|I| = 2g$, of the hyperplanes generated by $2g$ of the $p_i$.
\end{Lemma}
\begin{proof}
The faces of $\Eff(X^{2g}_1)$ are described in \cite[Lemma 4.24]{CT06}, and in \cite[Corollary 2.5]{BDP16}. It is straightforward to compute the extremal rays of $\Eff(X^{2g}_1)$ by intersecting its faces.
\end{proof}

\begin{Lemma}\label{canodd}
The canonical sheaf of $\Sigma_{2g}\subset\mathbb{P}^N$ is isomorphic to $\mathcal{O}_{\Sigma_{2g}}(-1)$.
\end{Lemma}
\begin{proof}
Let $X^{2g}_1$ be the blow-up of $\mathbb{P}^{2g}$ at $2g+2$ general points $p_1,\dots,p_{2g+2}$, and let $h_g:X^{2g}_1\dasharrow \Sigma_{2g}$ be the birational map induced by the morphism $\widetilde{\mu}_g:X^{2g}_{g-1}\rightarrow \Sigma_{2g}$ in Proposition \ref{resodd}. As usual we will denote by $H$ the pull-back to $X_1^{2g}$ of the hyperplane section of $\mathbb{P}^{2g}$, and by $E_{1},\dots,E_{2g+2}$ the exceptional divisors. In order to conclude, it is enough to note that
$$-K_{X^{2g}_1}\sim (2g+1)H-(2g-1)\sum_{i=1}^{2g+2}E_i$$ 
is an element of the linear system $\mathcal{L}_{2g}$ inducing $\mu_g:\mathbb{P}^{2g}\dasharrow \Sigma_{2g}$, and to argue as in the proof of Lemma \ref{can}.
\end{proof}

As we have seen in \ref{small}, $\Pic(\Sigma_{2g})\cong \mathbb{Z}^{2g+3}$. In the next section our aim will be to describe the cones of curves and divisors of $\Sigma_{2g}$. 

\stepcounter{thm}
\subsection{Mori Dream Spaces and chamber decomposition}\label{MDS}
Let $X$ be a normal projective variety. 
We denote by $N^1(X)$ the real vector space of $\mathbb{R}$-Cartier divisors modulo numerical equivalence. 
The \emph{nef cone} of $X$ is the closed convex cone $\Nef(X)\subset N^1(X)$ generated by classes of 
nef divisors. 
The \emph{movable cone} of $X$ is the convex cone $\Mov(X)\subset N^1(X)$ generated by classes of 
\emph{movable divisors}. These are Cartier divisors whose stable base locus has codimension at least two in $X$.
The \emph{effective cone} of $X$ is the convex cone $\Eff(X)\subset N^1(X)$ generated by classes of 
\emph{effective divisors}.
We have inclusions:
$$
\Nef(X)\ \subset \ \overline{\Mov(X)}\ \subset \ \overline{\Eff(X)}
$$
We say that a birational map  $f: X \dasharrow X'$ into a normal projective variety $X'$  is a \emph{birational contraction} if its
inverse does not contract any divisor. 
We say that it is a \emph{small $\mathbb{Q}$-factorial modification} 
if $X'$ is $\mathbb{Q}$-factorial  and $f$ is an isomorphism in codimension one.
If  $f: X \dasharrow X'$ is a small $\mathbb{Q}$-factorial modification, then 
the natural pullback map $f^*:N^1(X')\to N^1(X)$ sends $\Mov(X')$ and $\Eff(X')$
isomorphically onto $\Mov(X)$ and $\Eff(X)$, respectively.
In particular, we have $f^*(\Nef(X'))\subset \overline{\Mov(X)}$.

\begin{Definition}\label{def:MDS} 
A normal projective $\mathbb{Q}$-factorial variety $X$ is called a \emph{Mori dream space}
if the following conditions hold:
\begin{enumerate}
\item[-] $\Pic{(X)}$ is finitely generated,
\item[-] $\Nef{(X)}$ is generated by the classes of finitely many semi-ample divisors,
\item[-] there is a finite collection of small $\mathbb{Q}$-factorial modifications
 $f_i: X \dasharrow X_i$, such that each $X_i$ satisfies the second condition above, and 
 $$
 \Mov{(X)} \ = \ \bigcup_i \  f_i^*(\Nef{(X_i)})
 $$
\end{enumerate}
\end{Definition}

The collection of all faces of all cones $f_i^*(\Nef{(X_i)})$ above forms a fan which is supported on $\Mov(X)$.
If two maximal cones of this fan, say $f_i^*(\Nef{(X_i)})$ and $f_j^*(\Nef{(X_j)})$, meet along a facet,
then there exists a commutative diagram: 
  \[
  \begin{tikzpicture}[xscale=1.5,yscale=-1.2]
    \node (A0_0) at (0, 0) {$X_i$};
    \node (A0_2) at (2, 0) {$X_{j}$};
    \node (A1_1) at (1, 1) {$Y$};
    \path (A0_0) edge [->,swap]node [auto] {$\scriptstyle{h_i}$} (A1_1);
    \path (A0_2) edge [->]node [auto] {$\scriptstyle{h_j}$} (A1_1);
    \path (A0_0) edge [->,dashed]node [auto] {$\scriptstyle{\varphi}$}  (A0_2);
  \end{tikzpicture}
  \]
where $Y$ is a normal projective variety, $\varphi$ is a small modification, and $h_i$ and $h_j$ are small birational morphisms
of relative Picard number one. 
The fan structure on $\Mov(X)$ can be extended to a fan supported on $\Eff(X)$ as follows. 

\begin{Definition}\label{MCD}
Let $X$ be a Mori dream space. We describe a fan structure on the effective cone $\Eff(X)$, called the \emph{Mori chamber decomposition}. We refer to \cite[Proposition 1.11]{HK00} and \cite[Section 2.2]{Ok16} for details.
There are finitely many birational contractions from $X$ to Mori dream spaces, denoted by $g_i:X\rightarrow Y_i$.
The set $\Exc(g_i)$ of exceptional prime divisors of $g_i$ has cardinality $\rho(X/Y_i)=\rho(X)-\rho(Y_i)$.
The maximal cones $\mathcal{C}$ of the Mori chamber decomposition of $\Eff(X)$ are of the form:
$$
\mathcal{C}_i \ = \Cone \ \Big(g_i^*\big(\Nef(Y_i)\big),\:\Exc(g_i)\Big)
$$
We call $\mathcal{C}_i$ or its interior $\mathcal{C}_i^{^\circ}$ a \emph{maximal chamber} of $\Eff(X)$.
\end{Definition}

Let $m>1$ be an integer. Let $X_{m+2}^m$ be the blow-up of $\mathbb{P}^m$ at $m+2$ general points. It is well-known that $X_{m+2}^m$ is a Mori Dream Space \cite[Theorem 1.3]{CT06}, \cite[Theorem 1.3]{AM16}. 
In what follows we will describe the cones of divisors of $X_{m+2}^m$, as well as its Mori chamber decomposition. Once this is done, we will concentrate on the case where $m=2g$ in order to use these results to compute the cone of curves of $\Sigma_{2g}$. We will consider the standard bases of $\Pic(X_{m+2}^m)_{\mathbb{R}}\cong\mathbb{R}^{m+3}$ given by the pull-back $H$ of the hyperplane section and the exceptional divisors $E_1,\dots,E_{m+2}$. 

\begin{thm}\label{thmcd}
Let $X_{m+2}^m$ be the blow-up of $\mathbb{P}^m$ at $m+2$ general points, and write the class of a general divisor $D\in \Pic(X_{m+2}^m)_{\mathbb{R}}$ as $D = yH + \sum_{i=1}^{m+2}x_iE_i$. Then
\begin{itemize}
\item[-] The effective cone $\Eff(X_{m+2}^m)$ is defined by 
$$
\left\{
\begin{array}{ll}
 y+x_i\geq 0 & i\in\{1,\dots,m+2\}\\ 
 my+\sum_{i=1}^{m+2}x_i\geq 0 & \\
 y\geq 0 & \\
 my + \sum_{i\in I}x_i\geq 0 & I\subseteq\{1,\dots,m+2\},\: |I| = m+1  
 \end{array}  
\right.
$$
\item[-] The Mori chamber decomposition of $\Eff(X_{m+2}^m)$ is defined by the hyperplane arrangement
\stepcounter{thm}

\begin{equation}\label{mcdeq}
\left\{
\begin{array}{ll}
(2-k)y-\sum_{i\in I}x_i=0 & \text{if}\ I\subseteq \{1,\dots,m+2\}, \: |I|=k-1\\ 
(m-k+1)y-\sum_{i\in I}x_i+\sum_{i=1}^{m+2}x_i=0 & \text{if}\ I\subseteq \{1,\dots,m+2\}, \: |I|=k
\end{array}  
\right.
\end{equation}
with $2\leq k\leq \frac{m+3}{2}$.
\item[-] The movable cone $\Mov(X_{m+2}^m)$ is given by 
$$
\left\{
\begin{array}{ll}
(m-1)y-\sum_{i\in I}x_i+\sum_{i=1}^{m+2}x_i\geq 0 & \text{if}\ I\subseteq \{1,\dots,m+2\}, \: |I|=2\\ 
x_i\leq 0 & i\in \{1,\dots,m+2\}\\
y+x_i\geq 0 & i\in\{1,\dots,m+2\}\\ 
my+\sum_{i=1}^{m+2}x_i\geq 0
\end{array}  
\right.
$$
\item[-] All small $\mathbb{Q}$-factorial modifications of $X$ are smooth. Let $\mathcal{C}$ and $\mathcal{C}'$ be two adjacent chambers of $\Mov(X)$, corresponding 
		to small $\mathbb{Q}$-factorial modifications of $X$, $f:X\dasharrow \widetilde{X}$ and $f':X\dasharrow\widetilde{X}'$, respectively. These chambers are separated by a hyperplane $H_I$ in (\ref{mcdeq}), with $3\leq k\leq \frac{m+3}{2}$ and $|I|\in \{k-1,k\}$. Assume that $\varphi(\mathcal{C})\subset (H_I \leq  k)$ and $\varphi(\mathcal{C}')\subset (H_I \geq  k)$. Then the birational map $f'\circ f^{-1}:\widetilde{X} \dasharrow \widetilde{X}'$ flips a $\P^{k-2}$ into a $\P^{m+1-k}$.
\item[-] Let $\mathcal{C}$ be a chamber of $\Mov(X)$, corresponding to small $\mathbb{Q}$-factorial modification $\widetilde{X}$ of $X$. Let $\sigma\subset \partial \mathcal{C}$ be a wall such that $\sigma\subset \partial \Mov(X)$, and let $f:\widetilde{X}\rightarrow Y$ be the corresponding elementary contraction. Then either $\sigma$ is supported on a hyperplane of the form $(y+x_i=0)$ or $(my+\sum_{i=1}^{m+2}x_i=0)$ and $f:\widetilde{X}\to Y$ is a $\P^1$-bundle, or $\sigma$ is supported on a hyperplane of the form $(x_i=0)$ or $(m-1)y-\sum_{i\in I}x_i+\sum_{i=1}^{m+2}x_i=0$, $I\subseteq \{1,\dots,m+2\}$, $|I|=2$, and $f:\widetilde{X}\to Y$ is the blow-up of a smooth point, and the exceptional divisor of $f$ is the image in $\widetilde{X}$ of either an exceptional divisor $E_{i}$ or a divisor of the form $H-\sum_{i\in I}E_i$ with $I\subset\{1,\dots,m+2\}$, $|I|=m$.
\end{itemize}
\end{thm}
\begin{proof}
The statement about the effective cone is in Lemma \ref{eff}. By \cite[Theorem 1.2]{Ok16} the inequalities for the movable cones and for the Mori chamber decomposition follow by removing from the analogous systems of inequalities for the blow-up of $\mathbb{P}^m$ in $m+3$ points in \cite[Theorem 1.3]{AM16} the inequalities involving the exceptional divisor $E_{m+3}$. Similarly, by \cite[Theorem 1.2]{Ok16} the last two claims follow easily from the last two items in \cite[Theorem 1.3]{AM16}. For the computation of the walls in the Mori chamber decomposition we refer also to \cite{Mu05}.
\end{proof}

\begin{Remark}
Theorem \ref{thmcd} allows us to find explicit inequalities defining the cones $\Eff(X^{m}_{m+2})$, $\Mov(X^{m}_{m+2})$, and $\Nef(\widetilde{X})$, for any small $\mathbb{Q}$-factorial modification $\widetilde{X}$ of $X^{m}_{m+2}$. For instance, $\Nef(X^{m}_{m+2})$ is given by 
$$
\left\{
\begin{array}{ll}
x_i\leq 0 & i\in\{1,\dots,m+2\}\\ 
y+x_i+x_j\geq 0 & i,j\in\{1,\dots,m+2\}, \: i\neq j
\end{array}  
\right.
$$
\end{Remark}

\stepcounter{thm}
\subsection{The Fano model of $X_{m+2}^m$}\label{FanoMod}
First, let us consider the case when $m = 2g$ is even. Then the anti-canonical divisor
$$-K_{X_{2g+2}^{2g}} \sim (2g+1)H-(2g-1)\sum_{i=1}^{2g+2}E_i$$
lies in the interior of the chamber $\mathcal{C}_{Fano}$ defined by 
\stepcounter{thm}
\begin{equation}\label{eqeven}
\left\{
\begin{array}{ll}
(1-g)y-\sum_{i\in I}x_i=0 & if\ I\subseteq \{1,\dots,m+2\}, \: |I|=g\\ 
gy-\sum_{i\in I}x_i+\sum_{i=1}^{m+2}x_i=0 & if\ I\subseteq \{1,\dots,m+2\}, \: |I|=g+1
\end{array}  
\right.
\end{equation}
that is the chamber in (\ref{mcdeq}) of Theorem \ref{thmcd} for $k = g+1$. By Theorem \ref{thmcd} the chamber $\mathcal{C}_{Fano}$ corresponds to a smooth small $\mathbb{Q}$-factorial modification $X_{Fano}^{2g}$ of $X_{2g+2}^{2g}$ whose nef cone is given by (\ref{eqeven}). By Lemma \ref{canodd} we get that 
\stepcounter{thm}
\begin{equation}\label{GITFano}
X_{Fano}^{2g} \cong \Sigma_{2g}
\end{equation}
If $m = 2g-1$ is odd the class
$$-K_{X_{2g+1}^{2g-1}} \sim 2gH-(2g-2)\sum_{i=1}^{2g+1}E_i$$
lies in the intersection of the hyperplanes
\stepcounter{thm}
\begin{equation}\label{eqodd}
\left\{
\begin{array}{ll}
(1-g)y-\sum_{i\in I}x_i=0 & I\subseteq \{1,\dots,2g+1\}, \: |I|=g\\ 
(g-1)y-\sum_{i\in I}x_i+\sum_{i=1}^{2g+1}x_i=0 & I\subseteq \{1,\dots,2g+1\}, \: |I|=g+1
\end{array}  
\right.
\end{equation}
The variety corresponding to the intersection (\ref{eqodd}) is a small non $\mathbb{Q}$-factorial modification of $X_{2g+1}^{2g-1}$, and by Lemma \ref{can} we can identify this variety with the GIT quotient $\Sigma_{2g-1}$.

\begin{Corollary}\label{numexrays}
The Mori cone $\NE(X_{Fano}^{2g})$ of $X_{Fano}^{2g}$ has $\binom{2g+2}{g}+\binom{2g+2}{g+1}= \binom{2g+3}{g+1}$ extremal rays.
\end{Corollary}
\begin{proof}
Since $X_{Fano}^{2g}$ is a Mori Dream Space, $\NE(X_{Fano}^{2g})$ is polyhedral and finitely generated. Furthermore, $\NE(X_{Fano}^{2g})$ is dual to $\Nef(X_{Fano}^{2g})$. Therefore, the number of extremal rays of $\NE(X_{Fano}^{2g})$ is equal to the number of faces of $\Nef(X_{Fano}^{2g})$. Now, the statement follows from (\ref{eqeven}).
\end{proof}

\stepcounter{thm}
\subsection{Fibrations on $X_{Fano}^{2g}$}\label{fibrations}
Let us stick to the situation where $m=2g$ is the dimension of the GIT quotient. By the isomorphism (\ref{GITFano}) we may identify the Fano model $X_{Fano}^{2g}$ with the GIT quotient $\Sigma_{2g}$, which in turn, by Remark \ref{GITHas} is isomorphic to the moduli space $\cM_{0,\widetilde{A}[2g+3]}$ with weights $\widetilde{A}[2g+3] = \left(\frac{2}
{2g+3},\dots,\frac{2}{2g+3}\right)$. Now, let us consider the moduli space $\cM_{0,B[2g+3]}$ with weights $B[2g+3] = \left(\frac{2}{2g+2},\dots,\frac{2}{2g+2}\right)$, and the reduction morphism $\rho_{\widetilde{A}[2g+3],B[2g+3]}:\cM_{0,B[2g+3]}\rightarrow \cM_{0,\widetilde{A}[2g+3]}$. Note that $k\frac{2}{2g+2}>1$ if and only if $k\geq g+2$, and $k\geq g+2$ implies that $k\frac{2}{2g+3}>1$. Therefore, \cite[Corollary 4.7]{Ha} yields that $\rho_{\widetilde{A}[2g+3],B[2g+3]}$ is an isomorphism, and we may identify $X_{Fano}^{2g}$ with the moduli space $\cM_{0,B[2g+3]}$.

Recall that by Remark \ref{GITHas} we may interpret $\Sigma_{2g-1}$ as the Hassett space $\cM_{0,\widetilde{B}[2g+2]}$ with $\widetilde{B}[2g+2] = \left(\frac{2}{2g+2},\dots,\frac{2}{2g+2}\right)$. Therefore, for any $i = 1,\dots,2g+3$ we have a forgetful morphism
$$\pi_i:\cM_{0,B[2g+3]}\cong  X_{Fano}^{2g}\rightarrow \cM_{0,\widetilde{B}[2g+2]}\cong \Sigma_{2g-1}$$ 
Now, consider the $g$-planes in $X_{Fano}^{2g}$ described in (\ref{confodd}). Note that in $\textbf{C}\cup \textbf{D}$ we have 
\stepcounter{thm}
\begin{equation}\label{numexray}
\binom{2g+2}{g+1}+\binom{2g+2}{g} = \binom{2g+3}{g+1}
\end{equation}
$g$-planes. From the modular point of view these $g$-planes parametrize configurations of $2g+3$ points $(\mathbb{P}^1,x_1,\dots,x_{2g+3})$ in $\mathbb{P}^1$ with $g+1$ points coinciding. Let us fix a marked point, say $x_{2g+3}$. For any choice of $g+1$ points in $(x_1,\dots,x_{2g+2})$ we get a $g$-plane $H$ in $\textbf{C}\cup \textbf{D}$ and its complementary $g$-plane $H^c$. For instance, if $H$ is given by $(x_1=\dots=x_{g+1})$ then $H^c$ is defined by $(x_{g+2}=\dots=x_{2g+2})$. Note that $H$ and $H^c$ intersects in a point representing the configuration $(x_1=\dots=x_{g+1},x_{g+2}=\dots=x_{2g+2},x_{2g+3})$, and that the morphism $\pi_{2g+3}:\cM_{0,B[2g+3]}\cong  X_{Fano}^{2g}\rightarrow \cM_{0,\widetilde{B}[2g+2]}\cong \Sigma_{2g-1}$ contracts $H\cup H^c$ to the singular point of $\Sigma_{2g-1}$ representing the configuration $(x_1=\dots=x_{g+1},x_{g+2},\dots,x_{2g+2})= (x_1,\dots,x_{g+1},x_{g+2}=\dots=x_{2g+2})$.  

Thanks to the modular interpretation of $\pi_{2g+3}$ we see that $\pi_{2g+3}^{-1}(p)\cong\mathbb{P}^1$ for any $p\in \Sigma_{2g-1}\setminus \Sing(\Sigma_{2g-1})$, while $\pi_{2g+3}^{-1}(p_j)$ is the union of two $g$-planes in $\textbf{C}\cup \textbf{D}$ intersecting in one point for any $p_j\in\Sing(\Sigma_{2g-1})$. Indeed 
$$\frac{1}{2}\binom{2g+2}{g+1} = \binom{2g+1}{g}$$
is exactly the number of singular points of $\Sigma_{2g-1}$. More generally for any $i = 1,\dots,2g+3$ we may consider the set $\{x_1,\dots,x_{i-1},x_{i+1},\dots,x_{2g+3}\}$, and for any subset $I$ of cardinality $g+1$ of $\{x_1,\dots,x_{i-1},x_{i+1},\dots,x_{2g+3}\}$ we have a $g$-plane $H_I$ defined by requiring the points marked by $I$ to coincide, and the complementary $g$-plane $H_{I^c}$ defined as the locus parametrizing configurations where the marked points in $I^c$ coincide. Then the morphism $\pi_i:\cM_{0,B[2g+3]}\cong  X_{Fano}^{2g}\rightarrow \cM_{0,\widetilde{B}[2g+2]}\cong \Sigma_{2g-1}$ contracts the unions $H\cup H^c$ to the singular points of $\Sigma_{2g-1}$.

For any $g$-plane $H_I$ in $\textbf{C}\cup \textbf{D}$ we will denote by $L_I$ the class of a line in $H_I$. Note that $L_I$ is the class of the $1$-dimensional boundary stratum of $\Sigma_{2g}$ corresponding to configurations where the points in $I$ coincide, other $g$ points coincide as well, and the remaining two points are different. 

\begin{Proposition}\label{exray}
The classes $L_I,L_{I^c}$ of lines in $H_I,H_{I^c}\cong\mathbb{P}^g$ described above generate the extremal rays of $\NE(X_{Fano}^{2g})$.
\end{Proposition}
\begin{proof}
From the above discussion we have that for any $L_I$ there is a complementary class $L_{I^c}$, and a forgetful morphism $\pi_i:\cM_{0,B[2g+3]}\cong  X_{Fano}^{2g}\rightarrow \cM_{0,\widetilde{B}[2g+2]}\cong \Sigma_{2g-1}$, such that the fibers of $\pi_i$ over $\Sigma_{2g-1}\setminus \Sing(\Sigma_{2g-1})$ are isomorphic to $\mathbb{P}^1$, and $H_I\cup H_{I^c}$ is contracted by $\pi_i$ onto a singular point of $\Sigma_{2g-1}$. Our aim is to compute the relative Mori cone $\NE(\pi_i) $ of the morphism $\pi_i$. Let $C\subset X_{Fano}^{2g}$ be a curve contracted by $\pi_i$. Then $C$ is either a fiber of $\pi_i$ over a smooth point of $\Sigma_{2g-1}$ or a curve in $H_{I}\cup H_{I^c}$. In the latter case the class of $C$ may be written as a combination with non-negative coefficients of the classes $L_I$ and $L_{I^c}$. 

Now, for simplicity of notation assume that $i = 2g+3$, and consider the surface $S\subseteq X_{Fano}^{2g}$ parametrizing configurations of points of the type $(x_1=\dots=x_g,x_{g+1},x_{g+2}=\dots=x_{2g+1},x_{2g+2},x_{2g+3})$. Then the image of $\pi_{2g+3|S}:S\rightarrow \Sigma_{2g-1}$ is the line $\Gamma$ passing through a singular point $p\in \Sigma_{2g-1}$ parametrizing the configurations $(x_1=\dots=x_g,x_{g+1},x_{g+2}=\dots=x_{2g+1},x_{2g+2},x_{2g+3})$. Furthermore, $\pi_{2g+3|S}^{-1}(q)\cong \mathbb{P}^1$ for any $q\in \Gamma\setminus\{p\}$, and $\pi_{2g+3|S}^{-1}(p)= L_I\cup L_{I^c}$. Therefore, the general fiber of $\pi_{2g+3}$ is numerically equivalent to $L_I + L_{I^c}$.

This means that $\NE(\pi_i)$ is generated by the classes $L_I,L_{I^c}$ for $\{x_1,\dots,x_{i-1},x_{i+1},\dots,x_{2g+3}\}$ with $|I| = g+1$. Finally, since by \cite[Proposition 1.14]{De01} the subcone $\NE(\pi_i)$ of $\NE(X_{Fano}^{2g})$ is extremal, we get that $L_I,L_{I^c}$ generate extremal rays of $\NE(X_{Fano}^{2g})$.
\end{proof}

\begin{thm}\label{ful}
The Mori cone of the GIT quotient $\Sigma_{2g}$ is generated by the classes $L_I,L_{I^c}$ of $1$-dimensional boundary strata.
\end{thm}
\begin{proof}
By Proposition \ref{exray} the classes $L_I,L_{I^c}$ generate extremal rays of $\NE(\Sigma_{2g})$. By (\ref{numexray}) these are $\binom{2g+3}{g+1}$ rays which by Corollary \ref{numexrays} is exactly the number of extremal rays of $\NE(\Sigma_{2g})$.
\end{proof}

\begin{Remark}\label{flips}
Consider the Mori chamber decomposition in Theorem \ref{thmcd}. Note that in order to go from the chamber corresponding to $\Nef(X^{2g}_{2g+2})$ to the chamber corresponding to $\Nef(X^{2g}_{Fano})$ we must cross the walls in (\ref{mcdeq}) for any $3\leq k\leq 2g+1$. Indeed by the modular description of the small modification $\psi:X^{2g}_{2g+2}\dasharrow X^{2g}_{Fano}$ we see that it factors as:
$$X_0 = X^{2g}_{2g+2} \stackrel{\psi_1}{\dasharrow} X_1\stackrel{\psi_2}{\dasharrow} X_2\stackrel{\psi_3}{\dasharrow}\dots\stackrel{\psi_{g-1}}{\dasharrow}X_{g-1} = X^{2g}_{Fano}$$
where $\psi_i:X_{i-1}\dasharrow X_{i}$ is the flip of the strict transform in $X_{i-1}$ of the $i$-planes in $\mathbb{P}^{2g}$ generated by $i+1$ among the blown-up points. These strict transforms are disjoint $i$-planes in $X_{i- 1}$, while the flipped locus in $X_i$ is a disjoint union of $(2g-1-i)$-planes. In particular the $\binom{2g+2}{g+1}+\binom{2g+2}{g} = \binom{2g+3}{g+1}$ extremal rays in (\ref{numexray}) correspond to the $\binom{2g+2}{g+1}$ $g$-planes coming as strict transforms of the $g$-planes in $\mathbb{P}^m$ generated by $g+1$ of the marked points, plus the $g$-planes that are the flipped loci of the $(g-1)$-planes in $\mathbb{P}^n$ generated by $g$ of the marked points.
\end{Remark}

\stepcounter{thm}
\subsection{The cone of moving curves of $\Sigma_{2g}$}
In this section we describe extremal rays of the cone of moving curves $\Mov_1(\Sigma_{2g})$ of $\Sigma_{2g}$. Recall that an irreducible curve $C$ on a projective variety $X$ is called a moving
curve if $C$ is a member of an algebraic family of curves covering a dense subset of $X$. By \cite[Theorems 2.2 and 2.4]{BDPP13} the cone of moving curves is dual to the cone of pseudoeffective divisor classes which is spanned by classes that appear as limits of sequences of effective $\mathbb{Q}$-divisors.

By Theorem \ref{thmcd} we have that $\Eff(\Sigma_{2g})$ is closed and therefore $\Mov_1(\Sigma_{2g})$ is the dual cone of $\Eff(\Sigma_{2g})$. In particular, by the description of $\Eff(X^{2g}_{2g+2})$ in Theorem \ref{thmcd} we get that $\Mov_1(\Sigma_{2g})$ has exactly $4g+6$ extremal rays. 

\begin{thm}\label{moving}
The cone of moving curves $\Mov_1(\Sigma_{2g})$ of $\Sigma_{2g}$ is generated by the classes of the fiber of the forgetful morphisms $\pi_i:\Sigma_{2g}\rightarrow\Sigma_{2g-1}$ for $i = 1,\dots,2g+3$, and the strict transforms in $\Sigma_{2g}$ of a general line in $\mathbb{P}^{2g}$ and of the degree $2g$ rational normal curves in $\mathbb{P}^{2g}$ passing through $2g+1$ of the blown-up points. 
\end{thm}
\begin{proof}
By the description of the facets of $\Eff(X^{2g}_{2g+2})$ in Theorem \ref{thmcd} we see that the $4g+6$ extremal rays of $\Mov_1(X^{2g}_{2g+2}) = \Eff(X^{2g}_{2g+2})^{\vee}$ are generated by the classes of the strict transforms in $\Sigma_{2g}$ of a general line in $\mathbb{P}^{2g}$ and of the degree $2g$ rational normal curves in $\mathbb{P}^{2g}$ passing through $2g+1$ of the blown-up points, together with the classes of the strict transform $L_i$ of a line through the blown-up point $p_i\in\mathbb{P}^{2g}$ for $i = 1,\dots,2g+2$, and of the strict transform of a degree $2g$ rational normal curve $C$ through $p_1,\dots,p_{2g+2}$. We will identify this last two classes of curves with fibers of the forgetful morphisms. 

Now, consider the following diagram
\[
  \begin{tikzpicture}[xscale=1.7,yscale=-1.4]
    \node (A0_0) at (0, 0) {$\overline{\mathcal{M}}_{0,2g+3}$};
    \node (A0_2) at (2, 0) {$\Sigma_{2g}$};
    \node (A1_0) at (0, 1) {$\overline{\mathcal{M}}_{0,2g+2}$};
    \node (A1_1) at (1, 1) {$X_{2g+2}^{2g}$};
    \node (A1_2) at (2, 1) {$\Sigma_{2g-1}$};
    \path (A0_0) edge [->,swap]node [auto] {$\scriptstyle{\widetilde{\pi}_i}$} (A1_0);
    \path (A0_0) edge [->]node [auto] {$\scriptstyle{f}$} (A1_1);
    \path (A1_1) edge [->,dashed]node [auto] {$\scriptstyle{\psi}$} (A0_2);
    \path (A0_2) edge [->]node [auto] {$\scriptstyle{\pi_i}$} (A1_2);
    \path (A0_0) edge [->]node [auto] {$\scriptstyle{\rho}$} (A0_2);
  \end{tikzpicture}
  \]
where $\psi: X^{2g}_{2g+2}\dasharrow \Sigma_{2g}$ is the sequence of flips in Remark \ref{flips}, $\rho:\overline{\mathcal{M}}_{0,2g+3}\rightarrow\Sigma_{2g}$ is the reduction morphism, $f:\overline{\mathcal{M}}_{0,2g+3}\rightarrow X^{2g}_{2g+2}$ is the Kapranov's blow-up morphism in Construction \ref{kblusym}, $\pi_i:\Sigma_{2g}\rightarrow\Sigma_{2g-1}$ is a forgetful morphism, and $\widetilde{\pi}_i:\overline{\mathcal{M}}_{0,2g+3}\rightarrow\overline{\mathcal{M}}_{0,2g+2}$ is the corresponding forgetful morphism on $\overline{\mathcal{M}}_{0,2g+3}$. Now, by \cite{Ka} the fibers of the forgetful morphism $\widetilde{\pi}_i:\overline{\mathcal{M}}_{0,2g+3}\rightarrow\overline{\mathcal{M}}_{0,2g+2}$ are either the strict transforms of the lines through a point $p_i$ or the strict transforms of the degree $2g$ rational normal curves through $p_1,\dots,p_{2g+2}$. Now, note that $\rho$, and hence $\psi$, map a general fiber of $\widetilde{\pi}_i$ onto a general fiber of $\pi_i$.

Finally, to conclude it is enough to observe that since the $L_i$ and $C$ generate extremal rays of $\Mov_1(X^{2g}_{2g+2})$ and $\psi: X^{2g}_{2g+2}\dasharrow \Sigma_{2g}$ is a sequence of flips of small elementary contractions \cite[Proposition 3.14]{Bar08} yields that the fibers of the forgetful morphisms $\pi_i:\Sigma_{2g}\rightarrow\Sigma_{2g-1}$ generates extremal rays of $\Mov_1(\Sigma_{2g})$.
\end{proof}

\stepcounter{thm}
\subsection{Analogy with moduli of weighted pointed curves}
Finally, we would like to stress another link between the GIT quotients we are studying and moduli of weighted curves. Consider the map:
\stepcounter{thm}
\begin{equation}\label{mcdhas}
\begin{array}{cccc}
\phi: &\Eff(X^{2g}_{2g+2})\subset\Pic(X^{2g}_{2g+2})_{\mathbb{Q}}& \longrightarrow & \mathbb{Q}^{2g+3}\\
      & (y,x_1,\dots,x_{2g+2}) & \longmapsto & (a_1,\dots,a_{2g+3}),
\end{array}
\end{equation}
where 
$$a_j = \frac{y+x_j}{(2g+1)y+\sum_{i=1}^{2g+2}x_i},\ \mathrm{for}\ j = 1,\dots,2g+2$$ 
and 
$$a_{2g+3} = 2-\sum_{i=1}^{2g+2}a_i$$ 
Note that $\phi$ maps $\Eff(X^{2g}_{2g+2})$ onto the hypercube $[0,1]^{2g+3}\subset\mathbb{R}^{2g+3}$. Furthermore, via $\phi$ the walls of the Mori chamber decomposition in (\ref{mcdeq}) translate into $\sum_{i\in I}a_i = 1$ for $I\subset\{1,\dots,2g+3\}$, with $|I|\in \{k-1,k\}$ and $2\leq k\leq \frac{2g+3}{2}$. These walls are exactly the ones used in \cite[Section 8]{Ha} to describe variations of some GIT quotients of products of $\P^1$ in terms of moduli of weighted curves. Therefore, interpreting the rational numbers $(a_1,\dots,a_{2g+3})$ as the weights of pointed curves, we have that the map $\phi$ in (\ref{mcdhas}) translates the Mori chamber decomposition of Theorem \ref{thmcd} into the GIT chamber decomposition from \cite[Section 8]{Ha}.

For instance, $\phi(-K_{X^3_{5}}) = \phi(4,-2,\dots,-2) = \left(\frac{1}{3},\dots,\frac{1}{3}\right)$ and by Remark \ref{GITHas} these are the weights of the Segre cubic threefold. Similarly, taking $D=3H-E_1-\dots-E_5$, which is ample on $X^{3}_5$, we have that $\phi(D) = \phi(3,-1,\dots,-1) = \left(\frac{2}{7},\dots,\frac{2}{7},\frac{4}{7}\right)$, and by Construction \ref{kblusym} the moduli space with these weights is isomorphic to $X^3_{5}$ itself. 

\section{The Hilbert polynomial of $\Sigma_m$}\label{degrees}
In this section, we use part of the preliminary results developed in the preceding ones to describe some projective geometry of the GIT quotients. In particular we make massive use once again of the linear systems on the projective space of Theorem \ref{Kum}, in order to compute the Hilbert polynomial and the degree of $\Sigma_{2g-1}$ and $\Sigma_{2g}$ in their natural embeddings in the spaces of invariants. 

For any integer $0\leq r \leq s-1$ and for any multi-index $I(r) = \{i_1,\dots,i_{r+1}\} \subseteq\{1,\dots,s\}$ set
\stepcounter{thm}
\begin{equation}\label{kdef}
k_{I(r)} = \max\{m_{i_1}+\dots+m_{i_{r+1}}-rd,0\}
\end{equation}
In \cite[Definition 3.2]{BDP16} the authors define the \textit{linear virtual dimension} of a linear system $\mathcal{L}=\mathcal{L}_{n,d}(m_1,\dots,m_s)$ on $\P^n$ as the number
\stepcounter{thm}
\begin{equation}\label{lexdim}
\binom{n+d}{d}+\sum_{r=0}^{s-1}\sum_{I[r]\subseteq\{1,\dots,s\}}(-1)^{r+1}\binom{n+k_{I[r]}-r-1}{n}-1
\end{equation}
Furthermore, they define the \textit{linear expected dimension} of $\mathcal{L}_{n,d}(m_1,\dots,m_s)$, denoted by $\ldim(\mathcal{L})$, as follows: if the linear system $\mathcal{L}_{n,d}(m_1,\dots,m_s)$ is contained in a linear system whose linear virtual dimension is negative then we set $\ldim(\mathcal{L}) = -1$, otherwise we define $\ldim(\mathcal{L})$ to be the maximum between the linear virtual dimension of $\mathcal{L}_{n,d}(m_1,\dots,m_s)$ and $-1$.
\begin{Proposition}\label{hilbgen}
Let $\phi_{\mathcal{L}}:X\dasharrow\mathbb{P}(H^0(\mathbb{P}^{n},\mathcal{L})^{*})$ be the rational map induced by the linear system $\mathcal{L}:=\mathcal{L}_{n,d}(m_1,\dots,m_s)$. Assume that $\phi_{\mathcal{L}}$ is birational, and let $X_{\mathcal{L}} = \overline{\phi_{\mathcal{L}}(\mathbb{P}^n)}$. If $s\leq n+2$ then the Hilbert polynomial of $X_{\mathcal{L}}$ is given by 
$$h_{X_{\mathcal{L}}}(t) = \binom{dt+n}{n}+\sum_{r=0}^{s-1}\sum_{I[r]\subseteq\{1,\dots,s\}} (-1)^{r+1}\binom{n+tk_{I[r]}-r-1}{n}$$
In particular $\deg(X_{\mathcal{L}})= d^n+\sum_{r=0}^{s-1}\sum_{I[r]\subseteq\{1,\dots,s\}} (-1)^{r+1}k_{I[r]}^n$.
\end{Proposition}
\begin{proof}
Polynomials of degree $t\in\mathbb{N}$ on $\mathbb{P}(H^0(\mathbb{P}^{n},\mathcal{L})^{*})$ correspond to degree $td$ polynomials on $\mathbb{P}^{n}$ vanishing with multiplicity $m_i$ at $p_i$. Therefore, the Hilbert polynomial of $X_{\mathcal{L}}$ is given by 
$$h_{X_{\mathcal{L}}}(t) = h^0(\mathbb{P}^{n},t\mathcal{L})$$
Now, since $t\mathcal{L}$ is effective for any $t\geq 0$, and $s \leq n+1$ \cite[Theorem 4.6]{BDP16} yields that 
$$h^0(\mathbb{P}^{n},t\mathcal{L}) = \ldim(t\mathcal{L})+1$$
where $\ldim(t\mathcal{L})$ is the linear expected dimension of $t\mathcal{L}$ defined by (\ref{lexdim}). To get the formula for the Hilbert polynomial it is enough to observe that, in the notation of (\ref{kdef}), for the linear system $t\mathcal{L}_g$ we have 
$$k_{I[r]}(t\mathcal{L}) = \max\{t(m_{i_1}+\dots+m_{i_{r+1}}-rd),0\}= tk_{I[r]}(\mathcal{L})$$
Finally, note that $h_{X_{\mathcal{L}}}(t)$ may be written as
$$h_{X_{\mathcal{L}}}(t) = \frac{d^n+\sum_{r=0}^{s-1}\sum_{I[r]\subseteq\{1,\dots,s\}}k_{I[r]}^n}{n!}t^{n}+P(t)$$
where $P(t)$ is a polynomial in $t$ of degree $\deg(P)\leq n-1$. Therefore, the volume of the big linear system $\mathcal{L}$ is given by
$$\Vol(\mathcal{L}) = \limsup_{t\mapsto +\infty}\frac{h_{X_{\mathcal{L}}}(t)}{t^{n}/n!}= d^n+\sum_{r=0}^{s-1}\sum_{I[r]\subseteq\{1,\dots,s\}}k_{I[r]}^n$$
and $\Vol(\mathcal{L})$ is exactly the degree $X_{\mathcal{L}}\subseteq\mathbb{P}(H^0(\mathbb{P}^{n},\mathcal{L})^{*})$.
\end{proof}

Recall from Section \ref{oddismi} the we define $\mathcal{L}_{2g-1} = \mathcal{L}_{2g-1,g}(g-1,\dots,g-1)$ as the linear system of degree $g$ forms on $\mathbb{P}^{2g-1}$ vanishing with multiplicity $g$ at $2g+1$ general points $p_{1},\dots,p_{2g+1}\in\mathbb{P}^{2g-1}$ and that it induces a birational map $\sigma_g:\mathbb{P}^{2g-1}\dasharrow\mathbb{P}(H^0(\mathbb{P}^{2g-1},\mathcal{L}_{2g-1})^{*}).$ The GIT quotient $\Sigma_{2g-1}$ is the closure of the image of $\sigma_g$ in $\mathbb{P}(H^0(\mathbb{P}^{2g-1},\mathcal{L}_{2g-1})^{*})$. 

\begin{Corollary}\label{hilbseg}
The Hilbert polynomial of $\Sigma_{2g-1}\subseteq \mathbb{P}(H^0(\mathbb{P}^{2g-1},\mathcal{L}_{2g-1})^{*})$ is given by 
$$h_{\Sigma_{2g-1}}(t) = \binom{gt+2g-1}{2g-1}+\sum_{r=0}^{g-2}(-1)^{r+1}\binom{2g+1}{r+1}\binom{t(g-r-1)+2g-1-r-1}{2g-1}$$
In particular
$$h^0(\mathbb{P}^{2g-1},\mathcal{L}_{2g-1}) = \binom{3g-1}{2g-1}+\sum_{r=0}^{g-2}(-1)^{r+1}\binom{2g+1}{r+1}\binom{3g-2r-3}{2g-1}$$
and $\deg(\Sigma_{2g-1})= g^{2g-1}+\sum_{r=0}^{g-2}(-1)^{r+1}\binom{2g+1}{r+1}(g-r-1)^{2g-1}$.
\end{Corollary}
\begin{proof}
By Theorem \ref{Kum} $\Sigma_{2g-1} = \overline{\sigma_g(\mathbb{P}^{2g-1})}$, where $\sigma_g$ is the map in (\ref{mapseg}). Then polynomials of degree $t\in\mathbb{N}$ on $\mathbb{P}(H^0(\mathbb{P}^{2g-1},\mathcal{L}_{2g-1})^{*})$ correspond to degree $tg$ polynomials on $\mathbb{P}^{2g-1}$ vanishing with multiplicity $t(g-1)$ at $p_1,\dots,p_{2g+1}$. Note that in the notation of (\ref{kdef}) for the linear system $t\mathcal{L}_{2g-1}$ we have 
$$
k_{I[r]}=
\left\lbrace\begin{array}{ll}
t(g-r-1) & \rm{if}\quad \textit{r}\leq \textit{g}-2\\ 
0 & \rm{if} \quad \textit{r}\geq \textit{g}-1
\end{array}\right. 
$$
Furthermore, note that in (\ref{lexdim}) we have
$$\binom{n+k_{I[r]}-r-1}{n} = \binom{2g-1+t(g-r-1)-r-1}{2g-1}\neq 0$$
for $t\gg 0$ if and only if $r\leq g-2$. 

Now, in order to conclude it is enough to observe that for any $r=0,\dots,g-2$ we have $\binom{2g+1}{r+1}$ subsets of $\{1,\dots,2g+1\}$ of the form $I[r]$. Then the formulas for the Hilbert polynomial and the degree of $\Sigma_{2g-1}$ follow from Proposition \ref{hilbgen}. In particular, the dimension of the linear system $\mathcal{L}_{2g-1}$ is then given by $h^0(\mathbb{P}^{2g-1},\mathcal{L}_{2g-1}) = h_{\Sigma_{2g-1}}(1)$.
\end{proof}

\begin{Corollary}\label{hilbodd}
Let us consider the GIT quotient $\Sigma_{2g}$ and let $n=2g+3$ be the number of points on $\P^1$ that it parametrizes.
The Hilbert polynomial of $\Sigma_{2g}$ is given by 
$$h_{\Sigma_{2g}}(t) = \binom{(2g+1)t+2g}{2g}+\sum_{r=0}^{g-1}(-1)^{r+1}\binom{2g+2}{r+1}\binom{t(2g-2r-1)+2g-r-1}{2g}$$
In particular $\deg(\Sigma_{2g})= (2g+1)^{2g}+\sum_{r=0}^{g-1}(-1)^{r+1}\binom{2g+2}{r+1}(2g-2r-1)^{2g}$.
\end{Corollary}
\begin{proof}
By Theorem \ref{Kum} $\Sigma_{2g}\subset \mathbb{P}(H^0(\mathbb{P}^{2g},\mathcal{L}_{2g})^{*})$
is the closure of the image of the rational map induced by the linear system $\mathcal{L}_{2g}$ of degree $2g+1$ hypersurfaces in $\mathbb{P}^{2g}$ with multiplicity $2g-1$ at $p_i$ for $i = 1,\dots,2g+2$. Now, to conclude it is enough to observe that for the linear system $t\mathcal{L}_{2g}$ we have
$$
k_{I[r]}=
\left\lbrace\begin{array}{ll}
t(2g-1-2r) & \rm{if}\quad \textit{r}\leq \lfloor\frac{2\textit{g}-1}{2}\rfloor\\ 
0 & \rm{if} \quad \textit{r}> \lfloor\frac{2\textit{g}-1}{2}\rfloor
\end{array}\right. 
$$
and to argue as in the proof of Corollary \ref{hilbseg}.
\end{proof}

\section{Symmetries of GIT quotients}\label{secaut}
In this section, exploiting the arrangements of linear spaces contained in $\Sigma_{m}$ we described in Section \ref{gSeg} and \ref{linsubsodd}, we manage to compute the automorphism group of the GIT quotients $\Sigma_{2g}$ and $\Sigma_{2g-1}$. In several cases, automorphisms of moduli spaces tend to be modular, in the sense that they can be described in terms of the objects parametrized by the moduli spaces themselves. See for instance, \cite{BM13}, \cite{Ma}, \cite{MaM14}, \cite{MaM}, \cite{FM17}, \cite{Ma16}, \cite{FM17b}, \cite{Lin11}, \cite{Ro71} for moduli spaces of pointed and weighted curves, \cite{BGM13}, \cite{AFKM19} for moduli spaces of bundles over a curve, \cite{BM16} for generalized quot schemes, and \cite{Mas18}, \cite{Mas20} for spaces of complete forms. We confirm this behavior also for the GIT quotients $\Sigma_{2g}$ and $\Sigma_{2g-1}$. We would like to stress that while the results in the above cited paper rely on arguments coming from birational geometry and moduli theory, in this case we use fairly different techniques based on explicit projective geometry. While for $\Sigma_{2g-1}$ we will rely on its singularities in the case of $\Sigma_{2g}$ we will make use of the description of its effective cone in Theorem \ref{thmcd}.

\stepcounter{thm}
\subsection{Automorphisms of $\Sigma_{2g-1}$}
Recall that the variety $\Sigma_{2g-1}$ parametrizing ordered configurations of $n=2g+2$ points on $\P^1$, with the democratic polarization is singular. We will need the following technical lemma.

\begin{Lemma}\label{tangcone}
Let $p\in \Sigma_{2g-1}\subset\mathbb{P}^N$ be a singular point. The tangent cone of $\Sigma_{2g-1}$ at $p$ is a cone with vertex $p$ over the Segre product $\mathbb{P}^{g-1}\times \mathbb{P}^{g-1}$. In particular, $\Sigma_{2g-1}$ has an ordinary singularity of multiplicity $\frac{(2g-2)!}{((g-1)!)^2}$ at $p\in \Sigma_{2g-1}$.
\end{Lemma} 
\begin{proof}
The statement follows from \cite[Lemma 4.3]{HMSV09}.
\end{proof}

\begin{Corollary}\label{linaut}
Any automorphism of the GIT quotient $\Sigma_{2g-1}\subset\mathbb{P}^N$ is induced by an automorphism of $\mathbb{P}^N$.
\end{Corollary} 
\begin{proof}
Recall that, by Proposition \ref{picsegre}, $\Pic(\Sigma_{2g-1})$ is torsion free. Let $\phi$ be an automorphism of $\Sigma_{2g-1}$. Then $\phi^{*}K_{\Sigma_{2g-1}}\sim K_{\Sigma_{2g-1}}$. Lemma \ref{can} yields that $\phi^{*}\mathcal{O}_{\Sigma_{2g-1}}(1)\sim \mathcal{O}_{\Sigma_{2g-1}}(1)$, that is $\phi$ is induced by an automorphism of $\mathbb{P}^N$.
\end{proof}
 
\begin{thm}\label{thaut}
The automorphism group of the GIT quotient $\Sigma_{2g-1}$ is the symmetric group on $2g+2$ elements:
$$\Aut(\Sigma_{2g-1})\cong S_{2g+2}$$
for any $g\geq 2$.
\end{thm} 
\begin{proof}
Let $p\in \Sigma_{2g-1}\subset\mathbb{P}^N$ be a singular point, and $H$ a $g$-plane inside $\Sigma_{2g-1}$ passing through $p$. Then $H$ is contained in the tangent cone of $\Sigma_{2g-1}$ at $p$, and by Lemma \ref{tangcone} $H$ must be a linear space generated by $p$ and a $(g-1)$-plane in the Segre embedding of $\mathbb{P}^{g-1}\times \mathbb{P}^{g-1}$, and such a $(g-1)$-plane must be either of the form $\{pt\}\times \mathbb{P}^{g-1}$ or of the form $\mathbb{P}^{g-1}\times\{pt\}$.

Assume that $H$ intersects the $g$-planes of the family $\textbf{A}_p$ in the lines $\left\langle p, q_{ij}\right\rangle$, where $q_{ij} = \alpha_i\cap \beta_j$ is a singular point of $\Sigma_{2g-1}$, and the $g$-planes of the family $\textbf{B}_p$ in $p$. Our aim is to prove that then $H$ must be one of the $\beta_j$.

By Proposition \ref{resseg}, the resolution $\widetilde{\sigma}_g$ has a modular interpretation in terms of the reduction morphism $\rho_{\widetilde{A}[2g+2],A[2g+2]}:\cM_{0,A[2g+2]}\rightarrow \cM_{0,\widetilde{A}[2g+2]}$. The only $g$-planes in $\Sigma_{2g-1}$ that are images of subvarieties contained in the exceptional locus of the blow-up $f:\cM_{0,A[2g+2]}\rightarrow\mathbb{P}^{2g-1}$ are the $\widetilde{\sigma}_g(E_I^{g-1})$ described in Section \ref{gplanes}. Therefore, we may assume that $H$ is the closure of the image via $\sigma_g$ of a $g$-dimensional variety $Z\subset\mathbb{P}^{2g-1}$. We may assume that $p = \sigma_g(\left\langle p_1,\dots,p_g\right\rangle)$, and consider $\Pi = \sigma_g(\left\langle p_{g+1},\dots,p_{2g+1}\right\rangle)$. Let $\xi_i = \sigma_g(\left\langle p_{g+1},\dots,\hat{p}_i,\dots,p_{2g+1}\right\rangle)$ be the other $g+1$ singular points of $\Sigma_{2g-1}$ lying on $\Pi$, and denote by $\pi_p:\mathbb{TC}_p\Sigma_g\dasharrow \mathbb{P}^{g-1}\times\mathbb{P}^{g-1}$ the projection from the tangent cone of $\Sigma_{2g-1}$ at $p$ onto its base, and let $\pi_i:\mathbb{P}^{g-1}\times\mathbb{P}^{g-1}\rightarrow\mathbb{P}^{g-1}$ be the projection onto the factors for $i = 1,2$. Since we are assuming that $H$ intersects the $g$-planes of the family $\textbf{A}_p$ in the lines $\left\langle p, q_{ij}\right\rangle$ we have that $H$ must be of the form $(\pi_1\circ\pi_p)^{-1}((\pi_1\circ\pi_p)(\xi_i))$ for some $i = g+1,\dots,2g+1$. Therefore, $Z$ is of the form $\left\langle p_1,\dots,p_g,p_i\right\rangle$ for some $i = g+1,\dots,2g+1$. Hence $\sigma_{g|Z}:Z\dasharrow H$ is the standard Cremona transformation of $\mathbb{P}^g$ and $H$ is one of the $\beta_i$.

Now, assume that $H$ is a $g$-plane inside $\Sigma_{2g-1}$ through $p$ intersecting the $g$-planes of the family $\textbf{B}_p$ in the line $\left\langle p, q_{ij}\right\rangle$, where $q_{i,j} = \alpha_i\cap \beta_j$ is a singular point of $\Sigma_{2g-1}$, and that it also intersects the planes of the family $\textbf{A}_p$ in $p$. By Proposition \ref{autcrem}, the standard Cremona transformation of $\mathbb{P}^{2g-1}$ induces an automorphism $\phi_{Cr}:\Sigma_{2g-1}\rightarrow \Sigma_{2g-1}$, set $q = \phi_{Cr}(p)$. Note that $\phi_{Cr}$ maps the family $\textbf{A}_p$ to the family $\textbf{B}_q$, and the family $\textbf{B}_p$ to the family $\textbf{A}_q$. Since in our argument $p$ is an arbitrary singular point of $\Sigma_{2g-1}$, proceeding as in the previous case we can show that $H$ must be one of the $\alpha_i$.

Clearly, since $\Sigma_{2g-1}\cong\cM_{0,\widetilde{A}[2g+2]}$ with symmetric weights $A[2g+2] = \left(\frac{1}{g+1},\dots,\frac{1}{g+1}\right)$, the symmetric group $S_{2g+2}$ acts on $\Sigma_{2g-1}$ by permuting the marked points. Our aim is now to show that $\Sigma_{2g-1}$ has at most $(2g+2)!$ automorphisms. 

Now, let $\phi:\Sigma_{2g-1}\rightarrow \Sigma_{2g-1}$ be an automorphism. If $p\in \Sigma_{2g-1}$ is a singular point then $\phi(p)$ must be a singular point as well. Therefore, we have at most $|\Sing(\Sigma_{2g-1})| = \binom{2g+1}{g}$ choices for the image of $p$.  

By Corollary \ref{linaut} $\phi$ is induced by a linear automorphism of the ambient projective space. Then $\phi$ must map $g$-planes through $p$ to $g$-planes through $\phi(p)$. In particular, since $\phi$ stabilizes $\Sing(\Sigma_{2g-1})$, it maps $\textbf{A}_p\cup \textbf{B}_p$ to $\textbf{A}_{\phi(p)}\cup \textbf{B}_{\phi(p)}$. In order to do this, we have the following $2((g+1)!)^2$ possibilities:
$$
\left\lbrace\begin{array}{c}
\alpha_{i,p}\mapsto \alpha_{j,\phi(p)} \\ 
\beta_{i,p}\mapsto \beta_{k,\phi(p)}
\end{array}\right.\quad
\rm{or}\quad
\left\lbrace\begin{array}{c}
\alpha_{i,p}\mapsto \beta_{j,\phi(p)} \\ 
\beta_{i,p}\mapsto \alpha_{k,\phi(p)}
\end{array}\right.
$$
Summing up, we have 
$$2((g+1)!)^2\binom{2g+1}{g} = (2g+2)!$$
possibilities. Now, assume that $\phi(p)=p$, and that $\phi$ maps $\alpha_{i,p}$ to $\alpha_{i,p}$, and $\beta_{i,p}$ to $\beta_{i,p}$ for any $i = 1,\dots,g+1$. Then $\phi$ must fix all the singular points $q_{i,j}$ determined by $\textbf{A}_p\cup \textbf{B}_p$.

Now, let us take into account one of these points, say $q_{g+1,g+1}$. Since $\phi$ fixes $g+2$ nodes in linear general position on $\alpha_{g+1}$, and $g+2$ nodes in linear general position on $\beta_{g+1}$ we have that $\phi$ is the identity on both $\alpha_{g+1}$ and $\beta_{g+1}$. On the other hand, $\alpha_{g+1}$ and $\beta_{g+1}$ are elements of the configuration $\textbf{A}_{q_{g+1}}\cup \textbf{B}_{q_{g+1}}$, therefore all the singular points of $\Sigma_{2g-1}$ determined by $\textbf{A}_{q_{g+1}}\cup \textbf{B}_{q_{g+1}}$ are fixed by $\phi$ as well. Proceeding recursively we see that $\phi$ must then fix all the singular points of $\Sigma_{2g-1}\subset\mathbb{P}^{N}$. Note that, with respect to the expression for $\sigma_g$ given in \ref{basis}, the points $[1:0\dots:0],[0:\dots:0:1],[1:\dots:1]\in \Sigma_{2g-1}\subset \mathbb{P}(H^{0}(\mathbb{P}^{2g-1},\mathcal{L}_{2g-1}^{*})) = \mathbb{P}^N$ are singular points.

Hence the automorphism of $\mathbb{P}^N$ inducing $\phi$ fixes $N+2$ points in linear general position. Therefore, it is the identity and then $\phi = Id_{\Sigma_{2g-1}}$.  
\end{proof}
 
\stepcounter{thm}
\subsection{Automorphisms of $\Sigma_{2g}$} 
Recall that $\Sigma_{2g}$ parametrizes ordered configurations of $n=2g+3$ points on $\P^1$, with the democratic polarization. By Theorem \ref{Kum}, $\Sigma_{2g}\subset \mathbb{P}(H^0(\mathbb{P}^{2g},\mathcal{L}_{2g})^{*})$
is the closure of the image of the rational map induced by the linear system $\mathcal{L}_{2g}$, see also Section \ref{evenio}. We denote by $\mu_g:\mathbb{P}^{2g}\dasharrow \Sigma_{2g}\subset \mathbb{P}(H^0(\mathbb{P}^{2g},\mathcal{L}_{2g})^{*}) = \mathbb{P}^N$ the rational map induced by this linear system.

\begin{Lemma}\label{linautodd}
Any automorphism of $\Sigma_{2g}\subset\mathbb{P}^N$ is induced by an automorphism of $\mathbb{P}^N$.
\end{Lemma} 
\begin{proof}
Recall that by \ref{small} $\Pic(\Sigma_{2g})\cong\mathbb{Z}^{2g+3}$. Then it is enough to argue as in the proof of Corollary \ref{linaut} using Lemma \ref{canodd} instead of Lemma \ref{can}.
\end{proof}

\begin{thm}\label{thautodd}
The automorphism group of the GIT quotient $\Sigma_{2g}$ is the symmetric group on $n$ elements:
$$\Aut(\Sigma_{2g})\cong S_{2g+3}$$
for any $g\geq 1$.
\end{thm} 
\begin{proof}
Let $\phi\in \Aut(\Sigma_{2g})$ be an automorphism. By the discussion in \ref{small}, $\phi$ induces a pseudo-automorphism, that is an automorphism in codimension two, 
$$\theta:=\psi^{-1}\circ\phi\circ\psi:X_1^{2g}\dasharrow X_1^{2g}$$ 
The pseudo-automorphism $\theta$ must preserve the set of the extremal rays of $\Eff(X_1^{2g})$. Let $D\subset X_1^{2g}$ be the union of the exceptional divisors $E_i$ and of the strict transforms of the hyperplanes generated by $2g$ of the $p_i$. By Lemma \ref{eff} any irreducible component of $D$ generates an extremal ray of $\Eff(X_1^{2g})$. Furthermore, any of these irreducible components is the unique element in its linear equivalence class. Therefore, $\theta$ must keep $D$ stable.

Set $\widetilde{D}=\psi(D)$. Then the automorphism $\phi$ stabilizes $\widetilde{D}$. As we did in Section \ref{linsubsodd}, we will now consider $g$-planes; check Section \ref{linsubsodd} for the needed definitions. Note that any $g$-plane $\gamma_i\in \textbf{C}$ is the intersection of $\binom{g+1}{g-1}$ divisors of type $\psi(H-\sum_{i\in I}E_i)$, and that any $g$-plane $\delta_i\in \textbf{D}$ is the intersection of $g$ divisors of type $\psi(E_i)$. Therefore $\phi$ stabilizes the configuration of $g$-planes $\textbf{C}\cup \textbf{D}$. 

Now, let $p\in \Sigma_{2g}$ be one of the distinguished points in \ref{distpoints}. Then $q = \phi(p)$ must be a distinguished point as well. Let us denote by $H_1,H_2$ the two $g$-planes intersecting in $p$, and by $\Pi_1,\Pi_2$ the two $g$-planes intersecting in $q$. Therefore, by \ref{distpoints} we have $\frac{(2g+3)!}{2((g+1)!)^2}$ choices for the image of $p$.

Let $p,p_1^1,\dots,p_{g+1}^1$ and $p,p_1^2,\dots,p_{g+1}^2$ be the distinguished points on $H_1$ and $H_2$ respectively. Similarly, we denote by $q,q_1^1,\dots,q_{g+1}^1$ and by $q,q_1^2,\dots,q_{g+1}^2$ the distinguished points on $\Pi_1$ and $\Pi_2$, respectively. Now, for the image of $p_1^1$ we have $2(g+1)$ possibilities, namely $q_1^1,\dots,q_{g+1}^1,q_1^2,\dots,q_{g+1}^2$. Once this choice is made, it is determined whether the image of $H_1$ via $\phi$ is either $\Pi_1$ or $\Pi_2$. Therefore, for the image of $p_2^1$ we have $g$ possibilities, for $p_3^1$ we have $g-1$ possibilities, and so on until we are left with just one possibility for $p_{g+1}^1$. Summing up, for the images of $p_1^1,\dots,p_{g+1}^1$ we have $2(g+1)!$ possibilities. Similarly, we have also $(g+1)!$ possibilities for the images of $p_1^2,\dots,p_{g+1}^2$. Finally, we have 
$$2(g+1)!(g+1)!\frac{(2g+3)!}{2((g+1)!)^2} = (2g+3)! = |S_{2g+3}|$$ 
possibilities. Now, assume that $\phi$ fixes the points $p,p_1^1,\dots,p_{g+1}^1, p_1^2,\dots,p_{g+1}^2$. Since by Corollary \ref{linaut} $\phi$ is induced by an automorphism of $\mathbb{P}^N$ and the points $p,p_1^1,\dots,p_{g+1}^1\in H_1$, $p,p_1^2,\dots,p_{g+1}^2\in H_2$ are in linear general position, then $\phi$ restricts to the identity on both $H_1$ and $H_2$.

Then, by the description of the configuration $\textbf{C}\cup \textbf{D}$ in \ref{confodd}, $\phi$ must fix all the distinguished points in $\textbf{C}\cup \textbf{D}$, and then Remark \ref{genposodd} yields that the automorphism $\phi$ must be the identity.
\end{proof}

\bibliographystyle{amsalpha}
\bibliography{Biblio}

\end{document}